\documentclass[10pt]{article}

\usepackage{microtype}
\usepackage{graphicx}
\usepackage{subfigure}
\usepackage{booktabs} 

\usepackage{hyperref}

\usepackage{amsmath,amsbsy,amsgen,amscd,amssymb,amsthm,amsfonts,stmaryrd}
\usepackage{mathtools}

\usepackage{bm}

\usepackage{microtype}      

\usepackage{hyperref}      
\usepackage{url}            

\usepackage[usenames,dvipsnames]{xcolor}
\usepackage{graphicx}

\graphicspath{{art/}}

\definecolor{dark-gray}{gray}{0.3}
\definecolor{dkgray}{rgb}{.4,.4,.4}
\definecolor{dkblue}{rgb}{0,0,.5}
\definecolor{medblue}{rgb}{0,0,.75}
\definecolor{rust}{rgb}{0.5,0.1,0.1}

\hypersetup{urlcolor=rust}
\hypersetup{citecolor=dkblue}
\hypersetup{linkcolor=dkblue}

\newtheorem{theorem}{Theorem}
\newtheorem{lemma}[theorem]{Lemma}
\newtheorem{corollary}[theorem]{Corollary}

\theoremstyle{definition}

\numberwithin{equation}{section}
\numberwithin{theorem}{section}

\renewcommand{\phi}{\varphi}

\newcommand{\dist}{\operatorname{dist}}
\newcommand{\abs}[1]{\vert #1 \vert}
\newcommand{\norm}[1]{\Vert #1 \Vert}

\DeclareFontFamily{OT1}{pzc}{}
\DeclareFontShape{OT1}{pzc}{m}{it}{<-> s * [1.200] pzcmi7t}{}
\DeclareMathAlphabet{\mathpzc}{OT1}{pzc}{m}{it}
\DeclareMathOperator{\supp}{supp}
\newcommand{\dom}[1]{\mathrm{dom}\left(#1\right)}

\newcommand{\State}{\STATE}
\newcommand{\For}{\FOR}
\newcommand{\EndFor}{\ENDFOR}

\makeatletter
\newtheorem*{rep@theorem}{\rep@title}
\newcommand{\newreptheorem}[2]{%
	\newenvironment{rep#1}[1]{%
		\def\rep@title{#2 \ref{##1}}%
		\begin{rep@theorem}}%
		{\end{rep@theorem}}}
\makeatother

\newreptheorem{theorem}{Theorem}
\newreptheorem{lemma}{Lemma}

\usepackage{natbib}
\usepackage{algorithm,algorithmic}
\usepackage[margin=1in]{geometry}

\renewcommand{\cite}[1]{\citep{#1}}
\newcommand\acknowledge[1]{
  \begingroup
  \renewcommand\thefootnote{}\footnote{#1}
  \addtocounter{footnote}{-1}
  \endgroup
}

\usepackage[nameinlink]{cleveref}
\crefname{lemma}{lemma}{lemmas}

\title{Almost surely constrained convex optimization}
\date{2017\\ December}
\author{Olivier Fercoq$^\dagger$
$\quad$ Ahmet Alacaoglu$^\ast$
$\quad$ Ion Necoara$^{\ddagger}$
$\quad$ Volkan Cevher$^\ast$ 
\acknowledge{This work was supported by the Swiss National Science Foundation (SNSF) under grant number $200021\_178865 / 1$. 
This project has received funding from the European Research Council (ERC) under the European Union's Horizon $2020$ research and innovation programme (grant agreement no $725594$ - time-data).~~~
~~~$^\ddagger$The work of I. Necoara has received support from the Executive Agency for Higher Education, Research and Innovation Funding (UEFISCDI), Romania: PNIII-P4-PCE-2016-0731, project ScaleFreeNet, no. 39/2017.}
\\ \\
$^\dagger$LTCI, T\'el\'ecom ParisTech, Universit\'e Paris-Saclay, France\\ 
$^\ddagger$Automatic Control and Systems Engineering Department, \\University Politehnica of Bucharest, Romania\\
$^\ast$LIONS, Ecole Polytechnique F\'ed\'erale de Lausanne, Switzerland \\
}

\date{}

\begin{document}

\maketitle

\begin{abstract} 
We propose a stochastic gradient framework for solving stochastic composite convex optimization problems with (possibly) infinite number of linear inclusion constraints that need to be satisfied almost surely. We use smoothing and homotopy techniques to handle constraints without the need for matrix-valued projections. We show for our stochastic gradient algorithm $\mathcal{O}(\log(k)/\sqrt{k})$ convergence rate for general convex objectives and $\mathcal{O}(\log(k)/k)$ convergence rate for restricted strongly convex objectives. 
These rates are known to be optimal up to logarithmic factors, even without constraints. 
We demonstrate the performance of our algorithm with numerical experiments on basis pursuit, a hard margin support vector machines and a portfolio optimization and show that our algorithm achieves state-of-the-art practical performance.
\end{abstract}

\section{Introduction}
In many machine learning applications, optimization problems involve stochasticity in objective functions or constraint sets.
Even though the problems with stochastic objective functions are well-studied in the literature, investigation of stochastic constraints seem to be rather scarce. 
In particular, we focus on the following stochastic convex optimization template:
\begin{align}\label{eq: prob1}
\min _{x\in\mathbb{R}^d} & \{ P(x) := F(x) + h(x)\} \\
& A(\xi) x \in b(\xi)  ~~ \xi \text{-almost surely} \nonumber
\end{align}
where $F(x) = \mathbb{E}_\xi \left[ f(x, \xi) \right]$ with a convex and smooth $f(\cdot, \xi)$ such that $\mathbb{E}\left[ \nabla f(x, \xi) \right] = \nabla F(x)$; and  $h: \mathbb{R}^d \to \mathbb{R}\cup$ $\{+\infty\}$ is a nonsmooth, proximable convex function. 

We seek to satisfy the stochastic linear inclusion constraints in \eqref{eq: prob1} \textit{almost surely}. We argue that this change is what sets \eqref{eq: prob1} apart from the standard stochastic setting in the literature. Indeed, we assume that $A(\xi)$ is a $m \times d$ matrix-valued random variable and $b(\xi) \subseteq \mathbb R^m$ is a random convex set. 
For the special case when $A(\xi)$ is an identity matrix, \eqref{eq: prob1} recovers optimization problems where the constraint set is the intersection of a possibly infinite number of sets.

Applications of almost surely constrained problems arise in many fields, such as machine learning, operations research, and mathematical finance.
Interesting cases include semi-infinite linear programming, sparse regression, portfolio optimization, classification, distributed optimization and streaming settings, and consensus optimization problems in standard constrained optimization where the access to full data is not possible~\cite{sonnenburg2006large,abdelaziz2007multi,nedic2018network,towfic2015stability}.

Particular instances of~\eqref{eq: prob1} involve primal support vector machines (SVM) classification and sparse regression which are central in machine learning~\cite{shalev2011pegasos,garrigues2009homotopy}
Due to the huge volume of data that is used for these applications, storing or processing this data at once is in general not possible.
Therefore, using these data points one by one or in mini batches in learning algorithms is becoming more important.
One direction that the literature focused so far is solving unconstrained formulations of these problems, successes of which are amenable to regularization parameters that needs to be tuned.
By presenting a method capable of solving~\eqref{eq: prob1} directly, we present a parameter-free approach for solving these problems.

The most popular method of choice for solving constrained stochastic optimization problems is projected stochastic gradient descent (SGD)~\cite{nemirovski2009robust}.
However, in the case where we have infinite number of constraints, it is not clear how to apply the projection step.
To remedy this issue, many methods utilize alternating projections to tackle stochastic constraints by viewing them as an intersection of possibly infinite sets~\cite{patrascu2017nonasymptotic} (see Section \ref{sec:related} for a detailed discussion). 
For the special case when $A(\xi)$ is a vector-valued random variable, projection methods are efficient.
However, in the general case, applying projection with matrix-valued $A(\xi)$ may clearly impose a serious computational burden per iteration.

In this work, we take a radically different approach and use Nesterov's smoothing technique~\cite{nesterov2005smooth} for almost sure constraints instead of applying alternating projections.
In doing so, we avoid the need for projections to the constraints.
We make use of the stochastic gradients of $f(x, \xi)$, proximal operators of simple nonsmooth component $h(x)$ and simple projections to the set $b(\xi)$.

In a nutshell, our analysis technique combines ideas of smoothing and homotopy in the stochastic gradient framework.
We extend the previous analysis on smoothing with homotopy~\cite{tran2018smooth} to stochastic optimization with infinitely many constraints.
To our knowledge, this is the first application of smoothing for stochastic constraints.
Our contributions can be summarized as follows:

\begin{itemize}
\setlength\itemsep{.1mm}
\item We provide a simple stochastic gradient type algorithm which does not involve projections with matrix-valued random variables or heuristic parameter tuning.
\item We prove $\tilde{\mathcal{O}}(1/\sqrt{k})$ convergence rate for general convex objectives.
\item We prove $\tilde{\mathcal{O}}(1/{k})$ convergence rate for restricted strongly convex objectives.
\item We include generalizations of our framework for composite optimization with general nonsmooth Lispchitz continuous functions in addition to indicator functions.
\item We provide numerical evidence and verify our theoretical results in practice.
\end{itemize}
 
\textbf{Roadmap.} We recall the basic theoretical tools that we utilize and and lay out the notation in Section~\ref{sec: prelim}. 
The algorithm and its convergence guarantees are presented in Section~\ref{sec: alg_conv}. 
Section~\ref{sec: ext} shows how our results can be used to recover and extend previous works.
We review the related literature and compare our method with the existing ones in Section~\ref{sec:related}.
We conclude by presenting the practical performance of our method on three different problem instances in Section~\ref{sec: num}.
Proofs of the theoretical results are deferred to the appendix.

\section{Preliminaries}\label{sec: prelim}
\textbf{Notation.~} We use $\| \cdot\|$ to denote Euclidean norm and $\langle \cdot, \cdot \rangle$ to denote Euclidean inner product. The adjoint of a continuous linear operator is denoted by $^\top$. 
We will write a.s. in place of "almost surely" in the sequel.

We define the distance function to quantify the distance between a point $x$ and set $\mathcal{K}$ as $\dist{(x, \mathcal{K})}= \inf_{z\in\mathcal{K}} \| x-z\|$.
Given a function $\phi$, we use $\partial \phi(x)$ to denote its subdifferential at $x$.
For a given set $\mathcal{K}$, we denote the indicator function of the set by $\delta_{\mathcal{K}}(x) = 0$, if $x\in\mathcal{K}$ and $\delta_{\mathcal{K}}(x) = +\infty$ otherwise.
We define the support function of the set $\mathcal{K}$ as $\supp_{\mathcal K}(x) = \sup_{y\in \mathcal{K}}\langle x, y \rangle$.
The domain of a convex function $f$ is $\dom f = \{x\;:\; f(x) < +\infty\}$.
We use $\tilde{\mathcal{O}}$ notation to suppress the logarithmic terms.

We define the proximal operator of the convex function $\phi$ as
\begin{equation*}
\text{prox}_{\phi}(z) = \arg \min_{x} \phi(x) + \frac{1}{2} \| x-z\|^2.
\end{equation*}
We say that $\phi$ is a proximable function if computing its proximal operator is efficient.

Given a function $f$ and $L>0$, we say that $f$ is $L$-smooth if $\nabla f$ is Lipschitz continuous, which is defined as
\begin{equation*}
\| \nabla f(x) - \nabla f(y) \| \leq L  \| x - y\|, ~~ \forall x, y \in \mathbb{R}^d.
\end{equation*}
We say that the function $f$ is $\mu>0$ strongly convex if it satisfies,
\begin{equation*}
f(x) \geq f(y) + \langle \nabla f(y), x-y \rangle + \frac{\mu}{2} \| x - y \|^2, ~~ \forall x,y \in \mathbb{R}^d,
\end{equation*}
and, we say that the function $f$ is $\mu$-restricted strongly convex if there exists $x_\star$ such that,
\begin{equation*}
f(x) \geq f(x_\star) + \frac{\mu}{2} \| x - x_\star \|^2, ~~\forall x\in\mathbb{R}^d.
\end{equation*}
It is known that restricted strong convexity is a weaker condition than strong convexity since it is implied by strong convexity along the direction of the solution~\cite{Necoara2017,bolte2017error}.

\textbf{Space of random variables.~}
We will consider in this paper random variables of $\mathbb R^m$
belonging to the space
\[
\mathcal Y = \{ (y(\xi))_\xi \;:\;  \mathbb E[\norm{y(\xi)}^2] < + \infty \}
\]
We shall denote by $\mu$ the probability measure of the random variable $\xi$,
endowed with the scalar product
\[
\langle y, z \rangle = \mathbb E[y(\xi)^\top z(\xi)] =  \int y(\xi)^\top z(\xi) \mu(d\xi).
\]
$\mathcal Y$ is a Hilbert space and its norm is $\norm{y} = \sqrt{\mathbb E[\norm{y(\xi)}^2]}$.

\textbf{Smoothing.~} We are going to utilize Nesterov's smoothing framework to process almost sure linear constraints. 
Due to~\cite{nesterov2005smooth}, a smooth approximation of a nonsmooth convex function $g$ can be obtained as
\begin{equation}
\label{eq:nes_sm}
g_{\beta}(z) = \max_{u} \langle u, z \rangle - g^\ast(u) - \frac{\beta}{2} \| u \|^2,
\end{equation}
where $g^\ast(u) = \sup_{z} \langle z, u \rangle - g(z)$ is the Fenchel-conjugate of $g$ and $\beta > 0$ is the smoothness parameter.
As shown in~\cite{nesterov2005smooth}, $g_{\beta}$ is convex and $1/\beta$-smooth.

For the special case of indicator functions, $g(x) = \delta_{b}(x)$, where $b$ is a given convex set, $g^\ast(x) = \supp_{b}(x)$ and the smooth approximation is given by $g_{\beta}(z) = \frac{1}{2\beta}\dist{(z, b)^2}$.
Smoothing the indicator function is studied in~\cite{tran2018smooth} for the case of deterministic optimization, which we extend to the stochastic setting in this work.

\textbf{Duality.~} 
We define the stochastic function
\begin{equation*}
g(A(\xi)x, \xi) = \delta_{b(\xi)}(A(\xi)x).
\end{equation*}
Using basic probability arguments, Problem \eqref{eq: prob1} can be written equivalently as:
\begin{align*}
\min _{x\in\mathbb{R}^d} &  \mathbb{E}[ f(x, \xi)] + h(x) + \mathbb{E}[g(A(\xi)x, \xi)] =: P(x) \!+\! G(Ax)
\end{align*}
where $P(x) = \mathbb E[f(x, \xi)] + h(x)$, $A: \mathbb R^d \to \mathcal Y$ is the linear operator such that $(Ax)(\xi) = A(\xi)x$ for all $x$ and $G : \mathcal Y \to \mathbb R \cup \{+\infty\}$ 
is defined by
\[
G(z) = \int \delta_{b(\xi)}(z(\xi)) \mu(d \xi).
\]
We will assume that 
\begin{equation}\label{eq:norma}
\|A\|_{2,\infty} = \sup_\xi \|A(\xi) \| < +\infty,
\end{equation}
so that $A$ is in fact continuous.
Note that assuming a uniform bound on $\|A(\xi)\|$ is not restrictive since we can replace $A(\xi)x \in b(\xi)$ by 
\begin{equation*}
A'(\xi)x = \frac{A(\xi)x}{\|A(\xi)\|} \in b'(\xi) = \frac{b(\xi)}{\|A(\xi)\|},
\end{equation*}
without changing the set of vectors $x$ satisfying the constraint, and projecting onto $b'(\xi)$ is as easy as  projecting onto $b(\xi)$.

For the case of stochastic constraints, we define the Lagrangian $\mathcal L: \mathbb R^d \times \mathcal Y \to \mathbb R \cup \{+\infty \}$ as
\begin{equation*}
\mathcal{L}(x, y) = P(x) + \int \langle A(\xi)x, y(\xi) \rangle - \supp_{b(\xi)}(y(\xi))\mu(d\xi).
\end{equation*}

Using the Lagrangian, one can equivalently define primal and dual problems as 
\begin{equation*}
\min_{x\in\mathbb{R}^d} \max_{y\in\mathcal{Y}} \mathcal{L}(x, y),
\end{equation*}
and
\begin{equation*}
\max_{y\in\mathcal{Y}} \min_{x\in\mathbb{R}^d} \mathcal{L}(x, y). 
\end{equation*}
Strong duality refers to values of these problems to be equal.
It is known that Slater's condition is a sufficient condition for strong duality to hold~\cite{bauschke2011convex}.
In the context of duality in Hilbert spaces, Slater's condition refers to the following:
\begin{equation*}
0 \in \text{sri}(\dom G - A(\dom P))
\end{equation*}
where $\text{sri}(\cdot)$ refers to the strong relative interior of the set~\cite{bauschke2011convex}.
 
\textbf{Optimality conditions.~} We denote by $(x_\star, y_\star) \in \mathbb R^d \times \mathcal Y$   a saddle point of $\mathcal{L}(x, y)$.
For the constrained problem, we say that $x$ is an $\epsilon$-solution if it satisfies the following objective suboptimality and feasibility conditions
\begin{equation}\label{eq:opt_cond}
|P(x) - P(x_\star)| \leq \epsilon, ~~~ \sqrt{\mathbb{E}\left[ \dist(A(\xi)x, b(\xi))^2 \right]} \leq \epsilon.
\end{equation}

\section{Algorithm \& Convergence}\label{sec: alg_conv}
We derive the main step of our algorithm from smoothing framework.
The problem in~\eqref{eq: prob1} is nonsmooth both due to $h(x)$ and the constraints encoded in $g(A(\xi)x, \xi)$.
We keep $h(x)$ intact since it is assumed to be proximable, and smooth $g$  to get
\begin{equation}\label{eq: smoothed_prob}
P_{\beta}(x) = \mathbb{E}\left[ f(x, \xi) \right] + h(x) + \mathbb{E}\left[ g_{\beta}(A(\xi)x, \xi) \right],
\end{equation}
where $g_{\beta}(A(\xi)x, b(\xi)) = \frac{1}{2\beta} \dist(A(\xi)x, b(\xi))^2$.
We note that $P_{\beta}(x)$ is $L(\nabla F) + \frac{\|A\|_{2,2}^2}{\beta}$-smooth where 
\begin{equation*}
\|A\|_{2,2} = \sup_{x \neq 0} \frac{\sqrt{\mathbb E[\norm{A(\xi)x}^2]}}{\norm{x}} \leq \| A\| _{2, \infty},
\end{equation*}
$\| A\| _{2, \infty}$ being defined in~\eqref{eq:norma}.
Note that~\eqref{eq: smoothed_prob} can also be viewed as a quadratic penalty (QP) formulation.

The main idea of our method is to apply stochastic proximal gradient (SPG)~\cite{rosasco2014convergence} iterations to~\eqref{eq: smoothed_prob} by using homotopy on the smoothness parameter $\beta$.
Our algorithm has a double loop structure where for each value of $\beta$, we solve the problem~\eqref{eq: smoothed_prob} with SPG upto some accuracy.
This strategy is similar to inexact quadratic penalty (QP) methods which are studied for deterministic problems in~\cite{lan2013iteration}.
In stark contrast to inexact QP methods, Algorithm~\ref{alg: A1} has explicit number of iterations for the inner loop which is determined by theoretical analysis, avoiding difficult-to-check stopping criteria for the inner loop in standard inexact methods.
We decrease $\beta$ to $0$ according to update rules from theoretical analysis to ensure the convergence to the original problem~\eqref{eq: prob1} rather than the smoothed problem~\eqref{eq: smoothed_prob}.

In Algorithm~\ref{alg: A1}, we present our stochastic approximation method for almost surely constrained problems (SASC, pronounced as "sassy").
We note that Case 1 refers to parameters for general convex case and Case 2 refers to restricted strongly convex case.
\begin{algorithm}[ht!]
\begin{algorithmic}
\State $x_0^0 \in \mathbb R^d$
\State $\alpha_0 \leq \dfrac{3}{4 L(\nabla f)}$, and $\omega > 1$
\State Case 1: $m_0 \in \mathbb N_*$.
\State Case 2: $m_0 \geq \frac{\omega}{\mu \alpha_0}$.
\For{$s \in \mathbb N$}
\State $m_s = \lfloor m_0 \omega^s \rfloor$, and $\beta_s = 4 \alpha_s \norm{A}^2_{2, \infty}$
\State Case 1: $\alpha_s = \alpha_0 \omega^{-s/2}$.
\State Case 2: $\alpha_s = \alpha_0 \omega^{-s}$.
\For{$k \in \{0, \ldots, m_s-1\}$}
\State Draw $\xi = \xi^s_{k+1}$, and define $z = A(\xi)x^s_k$.
\State {\small $D(x_k^s, \xi) := \nabla f(x^s_k, \xi) + A(\xi)^\top \nabla_z g_{\beta_s}(A(\xi)x^s_k, \xi)$}
\State $x^s_{k+1} = \text{prox}_{\alpha_s h}\left( x^s_k - \alpha_s D(x_k^s, \xi)\right)$
\EndFor
\State $\bar x^s = \frac{1}{m_s} \sum_{k=1}^{m_s}x_{k}^s$
\State Case 1: $x_0^{s+1} = x^s_{m_s}$.
\State Case 2: $x_0^{s+1} = \bar{x}^s$.
\EndFor
\State \textbf{return} $\bar x^s$
\end{algorithmic}
\caption{SASC}
\label{alg: A1}
\end{algorithm}

It may look unusual at first glance that in the restricted strongly convex
case, the step size $\alpha_s$ is decreasing faster than in the general convex case. The reason is that restricted strong convexity allows us to
decrease faster the smoothness parameter $\beta_s$, and that the step size 
is driven by the smoothness of the approximation.

We will present a key technical lemma which is instrumental in our convergence analysis.
It will serve as a bridge to transfer bounds on the smoothness parameter and an auxiliary function that we define in~\eqref{eq: smooth_gap_func} to optimality results in the usual sense for constrained optimization, $i.e.$~\eqref{eq:opt_cond}.
This lemma can be seen as an extension of Lemma 1 from~\cite{tran2018smooth} to the case of almost sure constraints.
We first define the auxiliary function that we are going to utilize, which we name as the smoothed gap function
\begin{equation}\label{eq: smooth_gap_func}
S_{\beta}(x) = P_{\beta}(x)-P(x_\star).
\end{equation}

\begin{lemma}\label{lem: smooth_gap_lemma}
Let $(x_\star, y_\star)$ be a saddle point of
\begin{equation*}
\min_{x\in\mathbb{R}^d} \max_{y\in\mathcal{Y}} \mathcal{L}(x, y),
\end{equation*}
and note that $S_{\beta}(x) = P_{\beta}(x)-P(x_\star) = P(x) - P(x_\star) + \frac{1}{2\beta} \int \dist{(A(\xi)x, b(\xi))^2 \mu(d\xi)}$.
Then, the following hold:
\begin{align}
&S_{\beta}(x) \geq -\frac{\beta}{2}\|y_\star\|^2, \nonumber \\
&P(x) - P(x_\star) \geq -\frac{1}{4\beta} \!\!\int \!\!\dist{(A(\xi)x, b(\xi))^2}\mu(d\xi) - \beta \| y_\star \|^2, \nonumber \\
&P(x) - P(x_\star) \leq S_{\beta}(x), \nonumber \\
&\int \dist{(A(\xi)x, b(\xi))^2}\mu(d\xi) \leq 4\beta^2\| y_\star\|^2 + 4\beta S_{\beta}(x). \nonumber
\end{align}
\end{lemma}

The simple message of Lemma~\ref{lem: smooth_gap_lemma} is that if an algorithm decreases the smoothed gap function $S_{\beta}(x)$ and $\beta$ simultaneously, then it obtains approximate solutions to~\eqref{eq: prob1} in the sense of~\eqref{eq:opt_cond}, $i.e.$ it decreases feasibility and objective suboptimality. 

The main technical challenge of applying SPG steps to problem~\eqref{eq: smoothed_prob} with homotopy stems from the stochastic term due to constraints, which is
\begin{equation}\label{eq: stoc_cons}
\mathbb{E}[g_{\beta}(A(\xi)x, \xi)],
\end{equation}
with $g_{\beta}(A(\xi)x, b(\xi)) = \frac{1}{2\beta} \dist(A(\xi)x, b(\xi))^2$.

Even though this term is in a suitable form to apply SPG, its variance bound and Lipschitz constant of its gradient becomes worse and worse as $\beta_k \to 0$.
A naive solution for this problem would be to decrease $\beta_k$ slowly, so that these bounds will increase slowly so that they can be dominated by the step size.
Due to Lemma~\ref{lem: smooth_gap_lemma} however, the rate of decrease of $\beta_k$ directly determines the convergence rate, so a slowly decaying $\beta_k$ would result in slow convergence for the method.
Our proof technique carefully balances the rate of $\beta_k$ and the additional error terms due to using stochastic gradients of~\eqref{eq: stoc_cons}, so that the optimal rate of SPG is retained even with constraints.

We are going to to present the main theorems in the following two sections for general convex and restricted strongly convex objecives, respectively.
The main proof strategy in Theorem~\ref{th: nonsc_th} and Theorem~\ref{th: sc_th} is to analyze the convergence of $S_{\beta}(x)$ and $\beta_k$ and use Lemma~\ref{lem: smooth_gap_lemma} to translate the rates to objective residual and feasibility measures.

\subsection{Convergence for General Convex Objectives}  
In this section, we present the convergence results for solving~\eqref{eq: prob1} where only general convexity is assumed for the objective $P(x)$.
\begin{theorem}\label{th: nonsc_th}
Assume $F$ is convex and $L(\nabla F)$ smooth, 
and $\exists \sigma_f$ such that $\mathbb E[\norm{\nabla f(x,\xi) - \nabla F(x)}^2] \leq \sigma_f^2$.
Denote $M_s = \sum_{l=0}^s m_l$. Let us set $\omega > 1$, $\alpha_0 \leq \frac{3}{4L(\nabla f)}$, $m_0 \in \mathbb{N}_\ast$, $m_s = \lfloor m_0 \omega^s \rfloor$, and $\beta_s =4\alpha_s \| A\|^2_{2, \infty}$. Then, for all s,
\begin{align*}
&\mathbb E[P(\bar{x}^s) - P(x_\star)] \leq \frac{C_1}{\sqrt{M_s}} \left[ C_2 + \frac{\log(M_s/m_0)}{\log(\omega)}C_3 \right] \\
& \mathbb E[P(\bar{x}^s) - P(x_\star)] \geq  -\frac{2C_4}{\sqrt{M_s}}\|y_\star\|^2- \frac{C_1}{\sqrt{M_s}} \left[ C_2 + \frac{\log(M_s/m_0)}{\log(\omega)}C_3 \right] \\
&\sqrt{\mathbb{E}\left[ \dist(A(\xi)\bar{x}^s, b(\xi))^2 \right]} \leq \frac{1}{\sqrt{M_s}} \bigg[ 2C_4\|y_\star\| + 2 \sqrt{C_1 C_4}\sqrt{C_2 + \frac{\log(M_s/m_0)}{\log(\omega)}C_3} \bigg]
\end{align*}
where $C_1 = \frac{\sqrt{m_0\omega}}{\alpha_0 (m_0 -1)\sqrt{\omega-1}}$, $C_2 = \frac{\|x_\star - x_0^0\|^2}{2}+2\alpha_0 m_0 \sigma_f^2$, $C_3 = 2\alpha_0^2 \|A\|_{2,\infty}^2 m_0 \|y_\star \|^2 + 2\alpha_0 m_0 \sigma_f^2$ and  $C_4 = 4 \alpha_0 \sqrt{m_0}\|A\|_{2,\infty}^2 \frac{\sqrt{\omega}}{\sqrt{\omega-1}}$.
\end{theorem}

Note that the $\mathcal{O}(1/\sqrt{k})$ rate is known to be optimal for solving~\eqref{eq: prob1} with SGD~\cite{polyak1992acceleration,agarwal2009information}.
In Theorem~\ref{th: nonsc_th}, we show that by handling infinite number of constraints without projections, we only lose a logarithmic factor from this rate.

\subsection{Convergence for Restricted Strongly Convex Objectives}
In this section, we assume $P(x)$ in~\eqref{eq: prob1} to be restricted strongly convex in addition to $F$ being $L(\nabla F)$ smooth. Note that requiring restricted strong convexity of $P(x)$ is substantially weaker than requiring strong convexity of component functions $f(x, \xi)$ or $h(x)$, see \cite{Necoara2017} for more details. In this setting, we have:
\begin{theorem}\label{th: sc_th}
	Assume $F$ is convex and $L(\nabla F)$ smooth, $P$ is $\mu$-restricted strongly convex
	and $\exists \sigma_f$ such that $\mathbb E[\norm{\nabla f(x,\xi) - \nabla F(x)}^2] \leq \sigma_f^2$.
	Denote $M_s = \sum_{l=0}^s m_l$. Let us set $\omega > 1$, $\alpha_0 \leq \frac{3}{4L(\nabla f)}$, $m_0 \geq \frac{\omega}{\mu \alpha_0}$, $m_s = \lfloor m_0 \omega^s \rfloor$, and $\beta_s =4\alpha_s \| A\|^2_{2, \infty}$. Then, for all s,
	\begin{align*}
	&\mathbb E[P(\bar{x}^s) - P(x_\star)]\leq  \frac{1}{M_s} \left[ D_1 + \frac{\log(M_s/m_0)}{\log(\omega)} D_2 \right] \\
	&\mathbb E[P(\bar{x}^s) - P(x_\star)] \geq 
	 -\frac{2 D_3}{M_s}\|y_\star\|^2 - \frac{1}{M_s} \left[ D_1 + \frac{\log(M_s/m_0)}{\log(\omega)}D_2 \right] \\
	&\sqrt{\mathbb{E}\left[ \dist(A(\xi)\bar{x}^s, b(\xi))^2 \right]} \leq 
	  \frac{1}{{M_s}} \left[ 2D_3\|y_\star\| + 2 \sqrt{D_3} \sqrt{D_1 + \frac{\log(M_s/m_0)}{\log(\omega)}D_2} \right]
	\end{align*}
	where $D_1 = \frac{\omega}{\omega-1}\frac{m_0}{\alpha_0(m_0-1)}\frac 12 \norm{x_0^0 - x_\star}^2 + 2 \alpha_0 m_0  \frac{\omega}{\omega-1} \sigma_f^2$, $D_2 = \frac{2m_0^2\alpha_0\omega}{(m_0-1)(\omega-1)}\Big(\|A\|^2_{2, \infty}\|y_\star\|^2 + \sigma_f^2\Big)$,
	$D_3 = 4 \alpha_0 {m_0}\|A\|^2_{2, \infty} \frac{{\omega}}{{\omega - 1}}$.
\end{theorem}

Similar comments to Theorem~\ref{th: nonsc_th} can be made for Theorem~\ref{th: sc_th}.
We obtain $\mathcal{O}(\log(k)/k)$ convergence rate for both objective residual and feasibility under restricted strong convexity assumption, which is optimal up to a logarithmic factor for solving~\eqref{eq: prob1} even without constraints.

\section{Extensions}\label{sec: ext}
In this section, we present basic extensions of our framework to illustrate its flexibility.

We can extend our method for solving problems considered in~\cite{ouyang2012stochastic}:
\begin{equation}\label{eq: nonsmooth_lips}
\min_{x\in\mathbb{R}^d} P_u(x) := \mathbb{E}\left[ f(x, \xi) + g(A(\xi)x, \xi) \right] + h(x),
\end{equation}
where the assumptions for $f$ and $h$ are the same as~\eqref{eq: prob1} and $g$ is not an indicator function, but is Lipschitz continuous in the sense that
\begin{equation*}
\vert g(x, \xi) - g(y, \xi) \vert \leq L_g \| x-y\|, \forall x,y\in\mathbb{R}^d, \forall \xi.
\end{equation*}
This assumption is equivalent to $\text{dom}(g^\ast)$ being bounded~\cite{bauschke2011convex}, where $g^\ast$ is the Fenchel-conjugate function of $g(\cdot, \xi)$.
This special case with $h(x)=0$ is studied in~\cite{ouyang2012stochastic} with the specific assumptions in this section.
Inspired by~\cite{nesterov2005smooth}, it has been shown in~\cite{ouyang2012stochastic}, that one has the following bound for the smooth approximation of $g(\cdot, \xi)$ in the sense of~\eqref{eq:nes_sm}
\begin{equation}
\label{eq:stoc_sm_bound}
\mathbb{E}[g(A(\xi)x, \xi)] \leq \mathbb{E}[g_{\beta}(A(\xi)x, \xi)] + \frac{\beta}{2}L_g^2.
\end{equation}
We illustrate that we can couple our main results with~\eqref{eq:stoc_sm_bound} to recover the guarantees of~\cite{ouyang2012stochastic} with the addition of the nonsmooth proximable term $h(x)$.
\begin{corollary}
\label{cor:lips}
Denote by $x_\star$ a solution of~\eqref{eq: nonsmooth_lips}.\\
(a) Under the same assumptions as Theorem~\ref{th: nonsc_th}, and Lipschitz continuous $g(\cdot, \xi)$, one has
\begin{align*}
\mathbb{E} [P_u(\bar{x}^s) - P_u(x_\star)] &\leq \frac{C_1}{\sqrt{M_s}} \left[ C_2 + \frac{\log(M_s/m_0)}{\log(\omega)}C_3 \right] + \frac{C_4}{\sqrt{M_s}}L_g^2,
\end{align*}
where the constants $C_1, C_2, C_3, C4$ are defined in Theorem~\ref{th: nonsc_th}.\\
(b) Under the same assumptions as Theorem~\ref{th: sc_th}, and Lipschitz continuous $g(\cdot, \xi)$, one has
\begin{align*}
\mathbb{E} [P_u(\bar{x}^s) - P_u(x_\star)] &\leq \frac{1}{M_s} \left[ D_1 + \frac{\log(M_s/m_0)}{\log(\omega)} D_2 \right] + \frac{D_3}{{M_s}}L_g^2,
\end{align*}
where the constants $D_1, D_2, D_3$ are defined in Theorem~\ref{th: sc_th}.
\end{corollary}
Lastly, we can combine the problem template in~\eqref{eq: prob1} with~\eqref{eq: nonsmooth_lips} to arrive at the problem
\begin{align*}
\min_{x\in\mathbb{R}^d} \mathbb{E}\left[ f(x, \xi) + g_1(A_1(\xi)x, \xi) \right] + h(x), \\
A_2(\xi)x \in b(\xi), \xi \text{-almost surely} \nonumber
\end{align*}
where $g_1$ is a Lipschitz continuous function and we have the same assumptions as~\eqref{eq: prob1} for almost sure constraints.
Arguments in Corollary~\ref{cor:lips} can be combined in a straightforward way with our results from Section~\ref{sec: alg_conv} for solving this template.

\section{Related Works}
\label{sec:related}

The most prominent work for stochastic optimization problems is stochastic gradient descent (SGD)~\cite{nemirovski2009robust,moulines2011non,polyak1992acceleration}. Even though SGD is very well studied, it only applies when there does not exist any constraints in the problem template~\eqref{eq: prob1}.
For the case of simple constraints, $i.e.$ $h(x) = \delta_{\mathcal{K}}(x)$ in~\eqref{eq: prob1} and almost sure constraints are not present, projected SGD can be used~\cite{nemirovski2009robust}.
However, it requires $\mathcal{K}$ to be a projectable set, which does not apply to the general template of~\eqref{eq: prob1} which would involve the almost sure constraints in the definition of $\mathcal{K}$.
In the case where $h(x)$ in~\eqref{eq: prob1} is a nonsmooth proximable function~\cite{rosasco2014convergence} studied the convergence of stochastic proximal gradient (SPG) method which utilizes stochastic gradients of $f(x, \xi)$ in addition to the proximal operator of $h(x)$. 
This method generalize projected SGD, however, they cannot handle infinitely many constraints that we consider in~\eqref{eq: prob1} since it is not possible to project onto their intersection in general.

A line of work that is known as alternating projections, focus on applying random projections for solving problems that are involving the intersection of infinite number of sets.
In particular, these methods focus on the following template
\begin{equation}\label{eq: cvx_feas}
\min_{x\in\mathbb{R}^d} \mathbb{E}\left[f(x, \xi)\right]: ~~ x \in {\mathcal B} (:= \cap_{\xi  \in \Omega} {\mathcal B}(\xi)).
\end{equation}
Here, the feasible set ${\mathcal B}$ consists of the intersection of a possibly infinite number of convex sets. The case when $f(x, \xi) = 0$ which corresponds to the convex feasibility problem is studied in~\cite{necoara2018randomized}. For this particular setting, the authors combine the smoothing technique with minibatch SGD, leading to a stochastic alternating projection algorithm having linear convergence. 

The most related to our work is~\cite{patrascu2017nonasymptotic} where the authors apply a proximal point type  algorithm with alternating projections.  The main idea behind~\cite{patrascu2017nonasymptotic} is to apply smoothing to $f(x, \xi)$ and apply  stochastic gradient steps to the smoothed function, which corresponds to a stochastic proximal point type of update, combined with alternating projection steps. 
The authors show $\mathcal{O}(1/\sqrt{k})$ rate for general convex objectives and $\mathcal{O}(1/k)$ for smooth and strongly convex objectives. 
For strongly convex objectives,~\cite{patrascu2017nonasymptotic} requires smoothness of the objective which renders their results not applicable to our composite objective function in~\eqref{eq: prob1}. 
In addition, they require strong convexity of the objective function while our results are valid for a more relaxed  strong convexity assumption.
Lastly,~\cite{patrascu2017nonasymptotic} assumes the projectability of individual sets, whereas in our case, the constraints $A(\xi) x \in b(\xi)$ might not be projectable unless $A(\xi)$ and $b(\xi)$ are of very small dimension since the projection involves solving a linear system at each iteration.

Stochastic forward-backward algorithms can also be applied to solve~\eqref{eq: prob1}. 
However, the papers introducing those very general algorithms focused on proving convergence and did not present convergence rates~\cite{bianchi2015stochastic,bianchi2017constant,salim2018random}.
There are some other works that focus on~\eqref{eq: cvx_feas}~\cite{wang2015random,mahdavi2013stochastic,yu2017online} where the authors assume the number of constraints is finite, which is more restricted than our setting.

In the case where the number of constraints in~\eqref{eq: prob1} is finite and the objective function is deterministic, Nesterov's smoothing framework is studied in~\cite{tran2018smooth, van2017smoothing, tran2018adaptive} in the setting of accelerated proximal gradient methods.
These methods obtain $\mathcal{O}(1/k)$ ($\mathcal{O}(1/k^2)$) rate when the number of constraints is finite and $F(x)$ is a (strongly) convex function whose gradient $\nabla F$ can be computed. 

Another related work is~\cite{ouyang2012stochastic} where the authors apply Nesterov's smoothing to~\eqref{eq: nonsmooth_lips}.
However, this work does not apply to~\eqref{eq: prob1}, due to the Lipschitz continuous assumption on $g(\cdot, \xi)$. 
Note that in our main template~\eqref{eq: prob1}, $g(\cdot, \xi) = \delta_{b(\xi)}(\cdot)$, which is not Lipschitz continuous.

\section{Numerical Experiments}\label{sec: num}

We present numerical experiments on a basis pursuit problem on synthetic data, a hard margin SVM problem on the \texttt{kdd2010, rcv1, news20} datasets from~\cite{chang2011libsvm} and a portfolio optimization problem on \texttt{NYSE, DJIA, SP500, TSE} datasets from~\cite{borodin2004can}.

\subsection{Sparse regression with basis pursuit on synthetic data}
In this section, we consider the basis pursuit problem which is widely used in machine learning and signal processing applications~\cite{donoho2006compressed,arora2018compressed}:
\begin{align}\label{eq:bp}
\min_{x \in \mathbb R^{d}} & \;\norm{x}_1 \\
\text{st: } & a^\top x = b,  a.s. \notag
\end{align}
where $a\in\mathbb{R}^d$, $b\in\mathbb{R}$.
\begin{figure*}[t!]
\centering
\includegraphics[width=0.4\columnwidth]{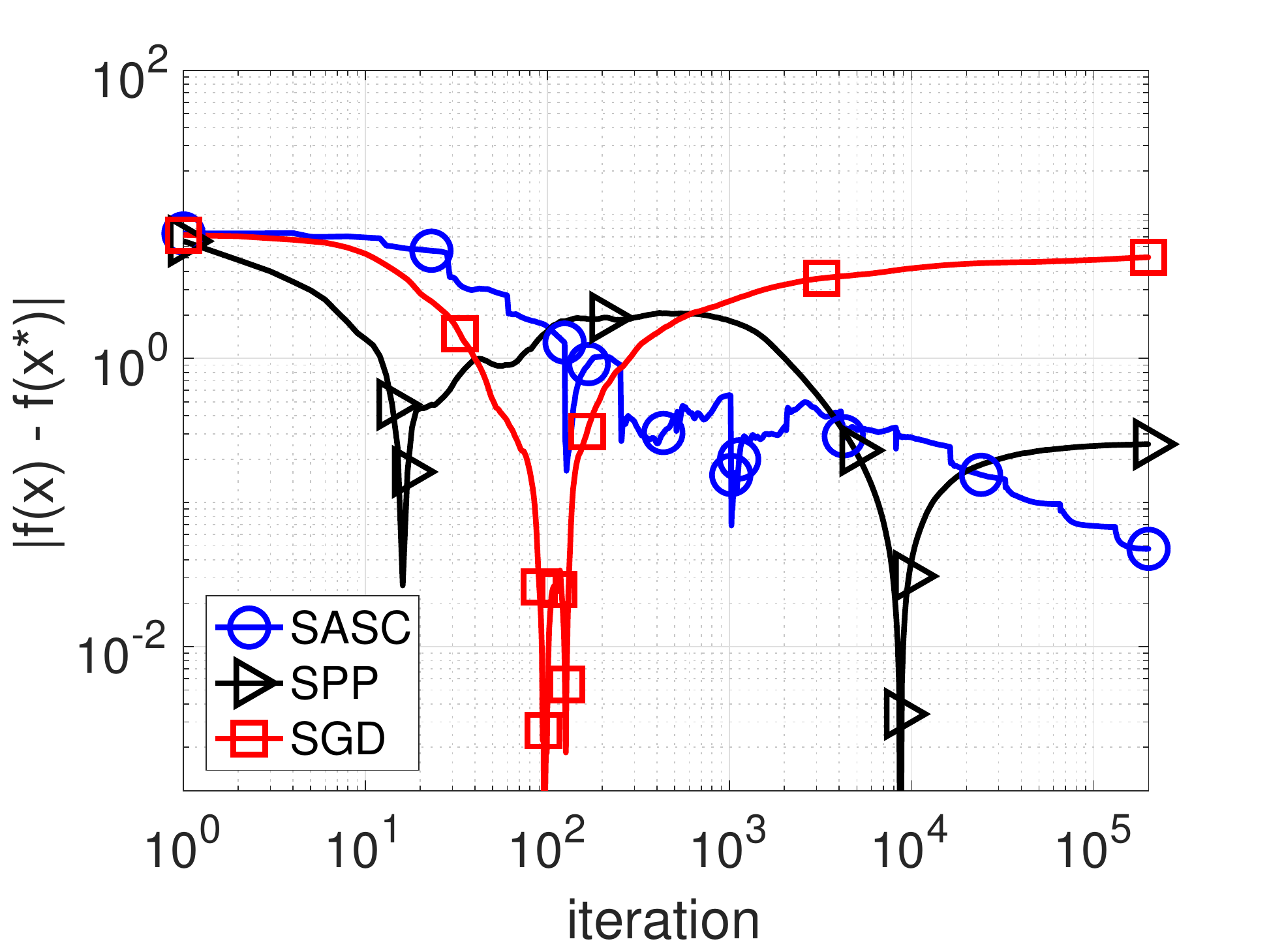}
\qquad 
\includegraphics[width=0.4\columnwidth]{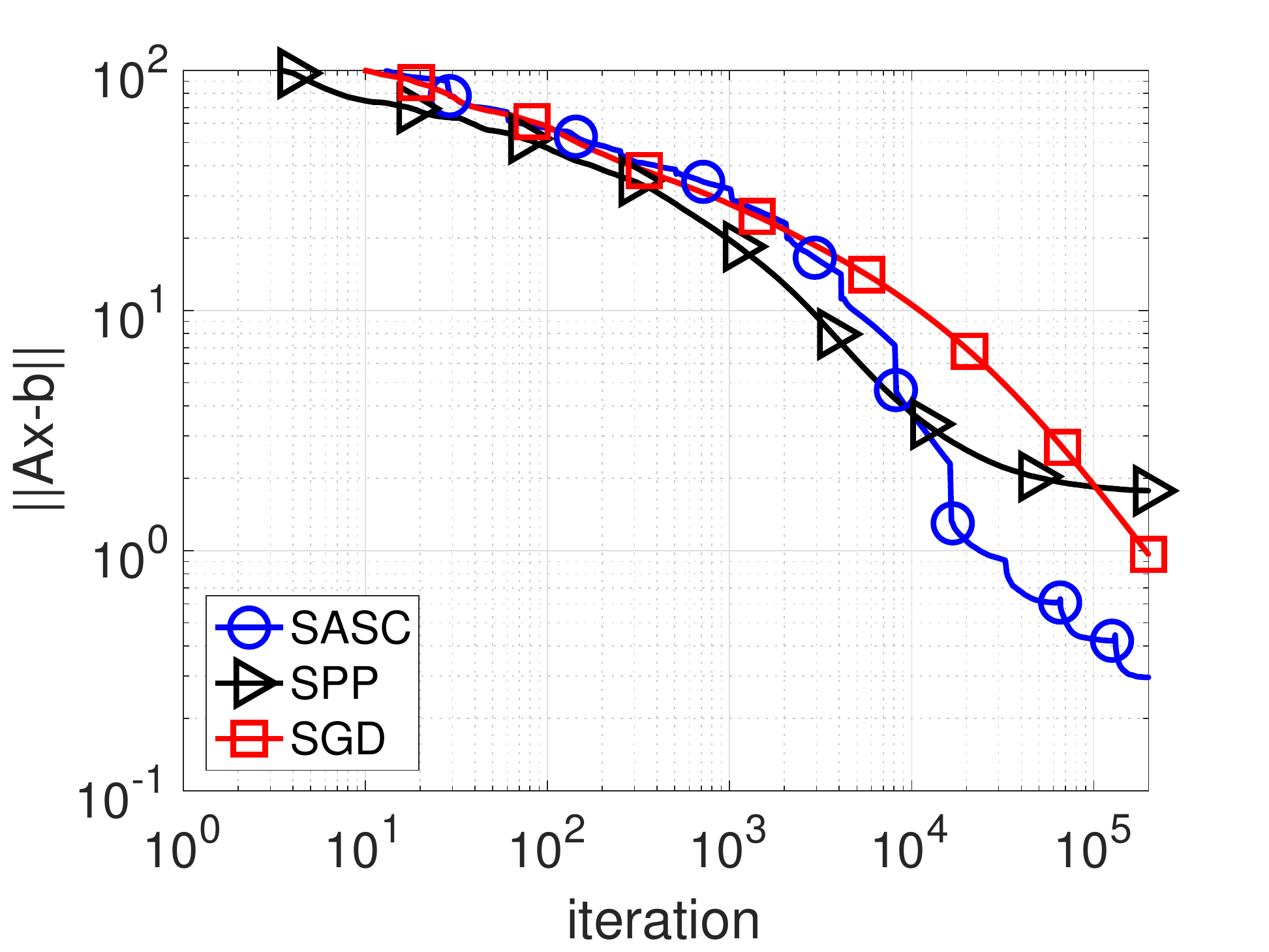}
\caption{Performance of SGD, SPP and SASC on synthetic basis pursuit problem.}
\label{fig:bp}
\end{figure*}
We consider the setting where the measurements $a$ arrive in a streaming fashion, similar to~\cite{garrigues2009homotopy}.
For generating the data, we defined $\Sigma$ as the matrix such that $\Sigma_{i,j} = \rho^{\abs{i-j}}$ with $\rho = 0.9$.
We generated a random vector $x^* \in \mathbb R^{d}$, $d=100$ with $10$ nonzero coefficients
and independent $\mathcal N(0, \Sigma)$ random variables $a_i$
which are then centered and normalized. We also define $b_i = a_i^\top x^*$.
Because of the centering, there are multiple solutions to the infinite system $a^\top x = b$ a.s., 
and we wish to recover $x^*$ as the solution of the basis pursuit problem~\eqref{eq:bp}.
We compare SASC (Algorithm~\ref{alg: A1}), SGD~\cite{nemirovski2009robust} and SPP~\cite{patrascu2017nonasymptotic}.
We manually tuned the step sizes for the methods and included the best obtained results. 
Since the basis pursuit problem does not possess (restricted) strong convexity, we use the parameters from Case 1 in SASC and a fixed step size $\mu$ for SPP which is used for the analysis in Corollary 6 in~\cite{patrascu2017nonasymptotic}.
We used the parameters $\mu=10^{-5}$ for SPP, $m_0 = 2$, $\omega=2$, $\alpha_0 = 10^{-2} \| a_1 b_1 \|_{\infty}$, where $a_1$ is the first measurement and $b_1$ is the corresponding result.
We take $n=10^5$ and make two passes over the data.
Figure~\ref{fig:bp} illustrates the behaviour of the algorithms for the synthetic basis pursuit problem. 
We can observe that SASC does exhibit a $\tilde O(1/\sqrt{k})$ convergence in feasibility and objective suboptimality. 
The stair case shape of the curves comes from the double-loop nature of the method. 
SPP can also solve this problem since the projection onto a hyperplane is easy to do when the constraints are processed one by one.
As observed in Figure~\ref{fig:bp}, SPP reaches to that accuracy almost as fast as SASC, however, it will stagnate once it reaches the pre-determined accuracy since the fixed step size $\mu$ determines the accuracy that the algorithm will reach.
We also tried running SGD on $\min_x \frac 12 \mathbb E(\norm{a^\top x- b}^2_2)$ but this leads to non-sparse solutions, therefore SGD converges to another solution than SASC and SPP.

A common technique that is used in stochastic optimization is to use mini-batches to parallelize and speed up computations.
Since SPP utilizes projections at each iteration, it needs to project onto linear constraints each iteration.
When the data is processed in mini-batches, this will require matrix inversions of sizes equal to mini-batches. On the other hand, SASC 
can handle mini-batches without any overhead.

\subsection{Portfolio optimization}
In this section, we consider Markowitz portfolio optimization with the task of maximizing the expected return given a maximum bound on the variance~\cite{abdelaziz2007multi}.
The precise formulation we consider is the following:
\begin{align}
\label{eq:portfolio}
\min_{x\in \mathbb{R}^d} -\langle a_{avg}, x \rangle:&
\sum_{i=1}^d x_i = 1 \\ &\vert \langle a_i - a_{avg}, x \rangle \vert \leq \epsilon, \forall i \in [1, n] \notag,
\end{align}
where short positions are allowed and $a_{avg} = \mathbb{E}[a_i]$ is assumed to be known. 

This problem fits to our template~\eqref{eq: prob1}, with a deterministic objective function, $n$ linear constraints and one indicator function for enforcing $\sum_{i=1}^d x_i = 1$ constraint.

We implement SASC and SPP from~\cite{patrascu2017nonasymptotic}.
Since the structure of~\eqref{eq:portfolio} does not have any restricted strong convexity due to linear objective function, we are applying the general convex version of SPP, which suggests setting a smoothness parameter $\mu$ depending on the final accuracy we would like to get as also discussed in basis pursuit problem.
We run SPP with two different $\mu$ values $10^{-1}$ and $10^{-2}$.
We run SASC with the parameters $\alpha_0 = 1$, $\omega=1.2$, $m_0=2$ and Case 1 in Algorithm~\ref{alg: A1}.
We use NYSE ($d = 36, n = 5651$), DJIA ($d = 30, n = 507$), SP500 ($d = 25, n = 1276$) and TSE ($d = 88, n = 1258$) where $d$ corresponds to the number of stocks and $n$ corresponds to the number of days for which the data is collected and we set $\epsilon$ in~\eqref{eq:portfolio} to be $0.2$.
These datasets are also used in~\cite{borodin2004can}.

We compute the ground truth using \texttt{cvx}~\cite{grant2008cvx} and plotted the distance of the iterates of the algorithms to the solution $\| x - x^\star\|$.
We compile the results in Figure~\ref{fig:portfolio}.

\begin{figure*}[h!]
\centering
\includegraphics[width=0.23\columnwidth]{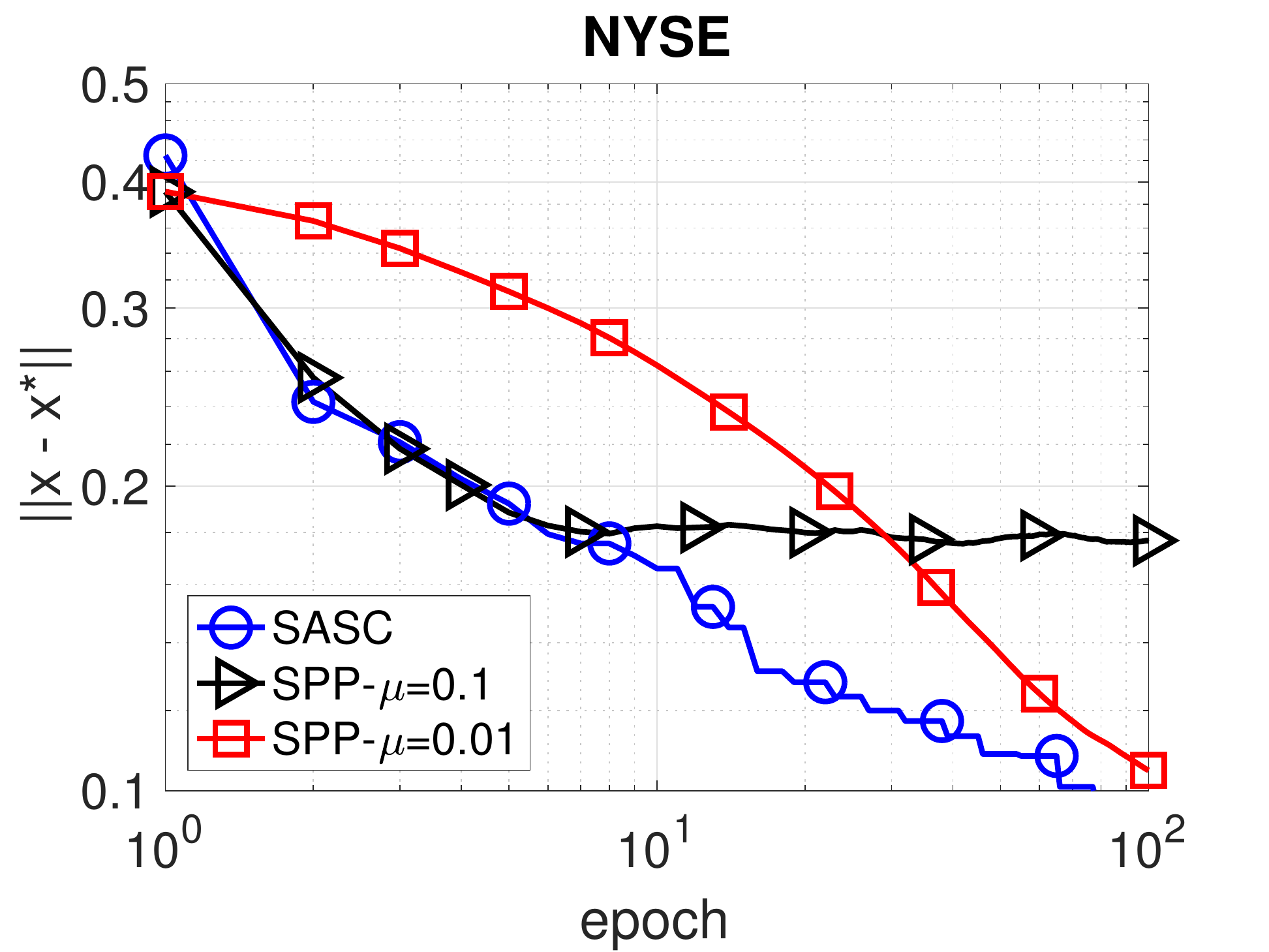}
\includegraphics[width=0.23\columnwidth]{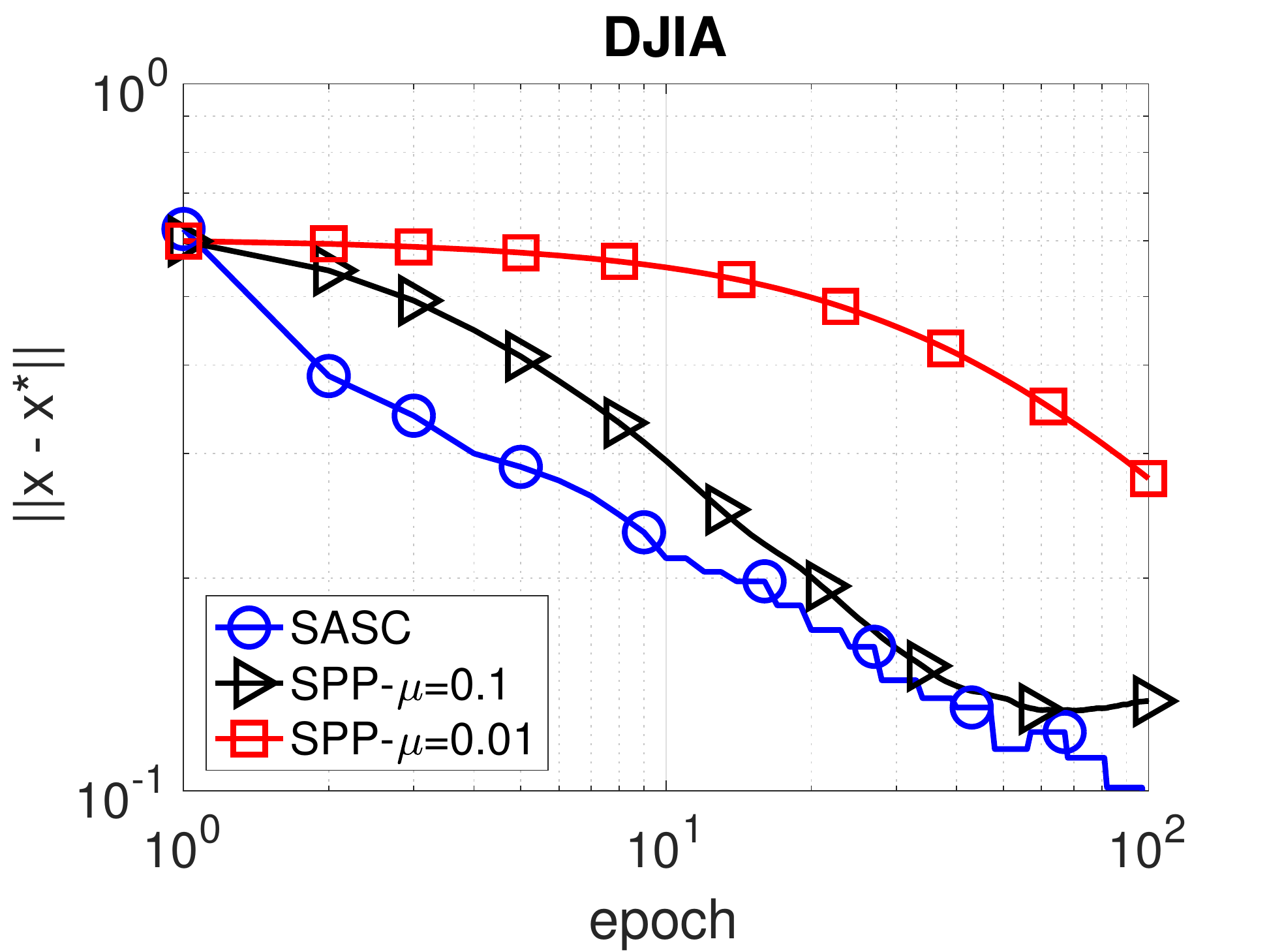}
\includegraphics[width=0.23\columnwidth]{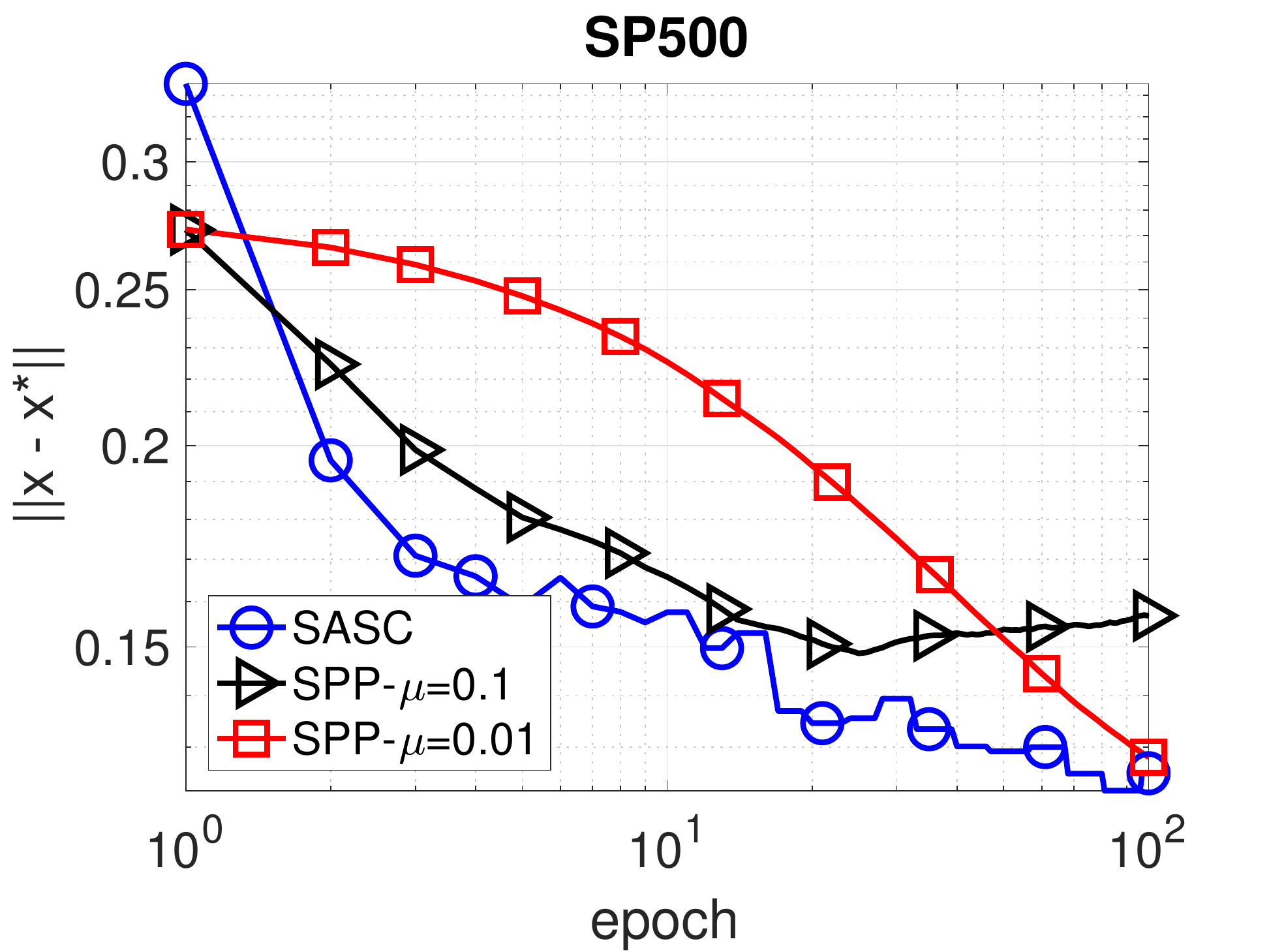}
\includegraphics[width=0.23\columnwidth]{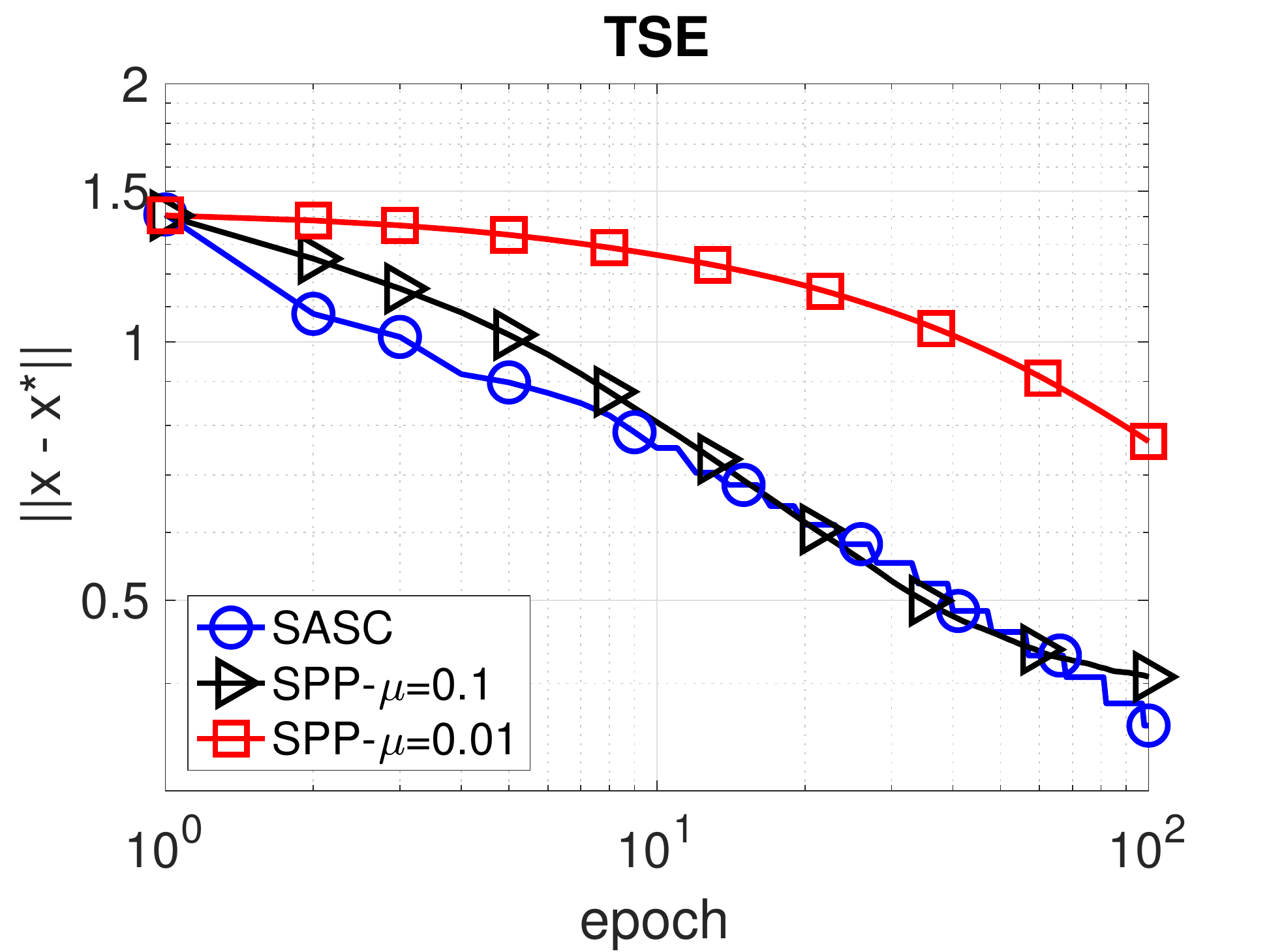}
\caption{Performance of SASC and SPP on portfolio optimization for four different datasets}
\label{fig:portfolio}
\end{figure*}

We can observe the behaviour of SPP from Figure~\ref{fig:portfolio} for different step size values $\mu$. 
Larger $\mu$ causes a fast decrease in the beginning, however, it also affects the accuracy that the algorithm is going to reach.
Therefore, large $\mu$ has the problem of stagnating at a low accuracy.
Smaller $\mu$ causes SPP to reach to higher accuracies at the expense of slower initial behaviour.
SASC has a steady behaviour and since it does not have a parameter depending on the final accuracy.
It removes the necessity of tuning $\mu$ in SPP, as we can observe the steady decrease of SASC throughout, beginning from the initial stage of the algorithm.

\subsection{Primal support vector machines without regularization parameter}
In this section, we consider the classical setting of binary classification, with a small twist.
For the standard setting, given a training set $\{ a_1, a_2, \dots, a_n \}$ and labels $\{ b_1, b_2, \dots, b_n \}$, where $a_i \in \mathbb{R}^p, \forall i$ and $b_i \in [-1, +1]$ the aim is to train a model that will classify the correct labels for the unseen examples.

Primal hard margin SVM problem is
\begin{equation}
\label{eq:svm_hm}
\min_{x\in\mathbb{R}^d} \frac{1}{2} \| x \|^2:
b_i \langle a_i, x \rangle \geq 1, \forall i.
\end{equation}
Since this problem does not have a solution unless the data is linearly separable, the standard way is to relax the constraints, and solve the soft margin SVM problem with hinge loss instead:
\begin{equation}
\label{eq:svm_uncons}
\min_{x\in\mathbb{R}^d} \frac{1}{2} \| x \|^2 + C \sum_{i=1}^n \max{\{ 0, 1-b_i \langle a_i, x\rangle\}},
\end{equation}
where $C$ has the role of a regularization parameter to be tuned.
The choice for $C$ has a drastic effect on the performance of the classifier as also been studied in the literature~\cite{hastie2004entire}.
It is known that poor choices of $C$ may lead to poor classification models.

We are going to have a radically different approach for the SVM problem.
Since the original formulation~\eqref{eq:svm_hm} fits to our template~\eqref{eq: prob1}, we can directly apply SASC to this formulation.
Even though the hard margin SVM problem does not necessarily have solution, applying SASC to~\eqref{eq:svm_hm} corresponds to solving a sequence of soft margin SVM problems with squared hinge loss, with changing regularization parameters.
The advantage of such an approach will be that there will be no necessity for a regularization parameter $C$ since this parameter will correspond to $\frac{1}{\beta}$ in our case where $\beta$ is the smoothness parameter, for which we have theoretical guideline from our analysis.

We compare SASC with Pegasos algorithm~\cite{shalev2011pegasos} which solves~\eqref{eq:svm_uncons} by applying stochastic subgradient algorithm.
Since the selection of the regularization parameter $C$ effects the performance of the model, we use 3 different values for the $\lambda$, namely $\{ \lambda_1, \lambda_2, \lambda_3 \} = \{ 10^{-3}/n, 1/n, 10^3/n \}$.
We use the following datasets from libsvm database~\cite{chang2011libsvm}: \texttt{kdd2010 raw version (bridge to algebra)} with $19,264,997$ training examples, $748,401$ testing examples and $1,163,024$ features, \texttt{rcv1.binary} with $20,242$ training examples, $677,399$ testing examples and $47,236$ features. For the last dataset, \texttt{news20.binary}, since there was not a dedicated testing dataset, we randomly split examples for training and testing with $17.996$ training examples, $2,000$ testing examples and $1,355,191$ features.
For SASC, we use $\alpha_0 = 1/2$, $\omega = 2$ in all experiments and use the parameter choices in Case 2 in Algorithm~\ref{alg: A1} due to strong convexity in the objective.
We computed the test errors for one pass over the data and compile the results in Figure~\ref{fig:svm}.

\begin{figure*}[h!]
\centering
\includegraphics[width=0.3\columnwidth]{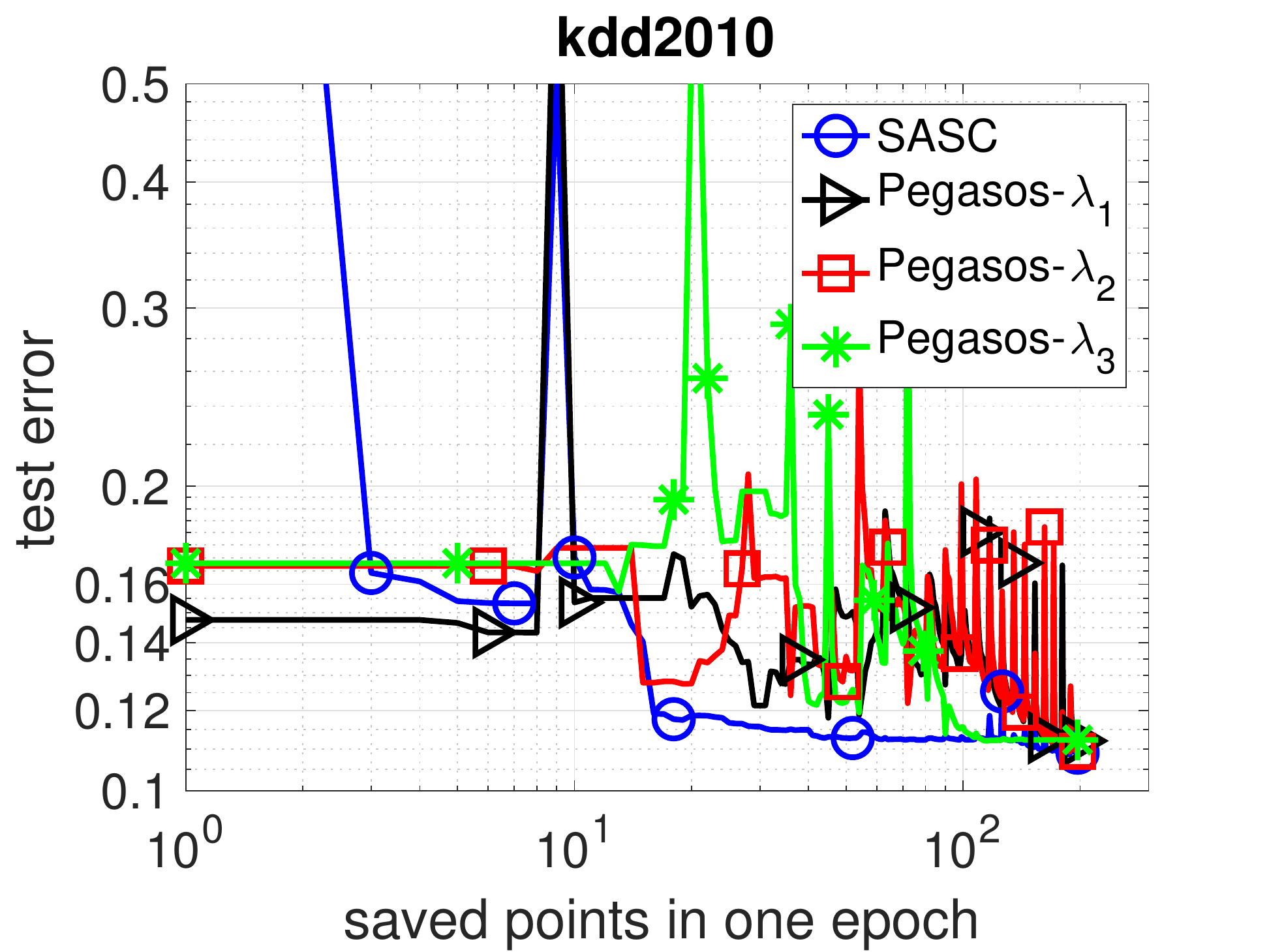}
\includegraphics[width=0.3\columnwidth]{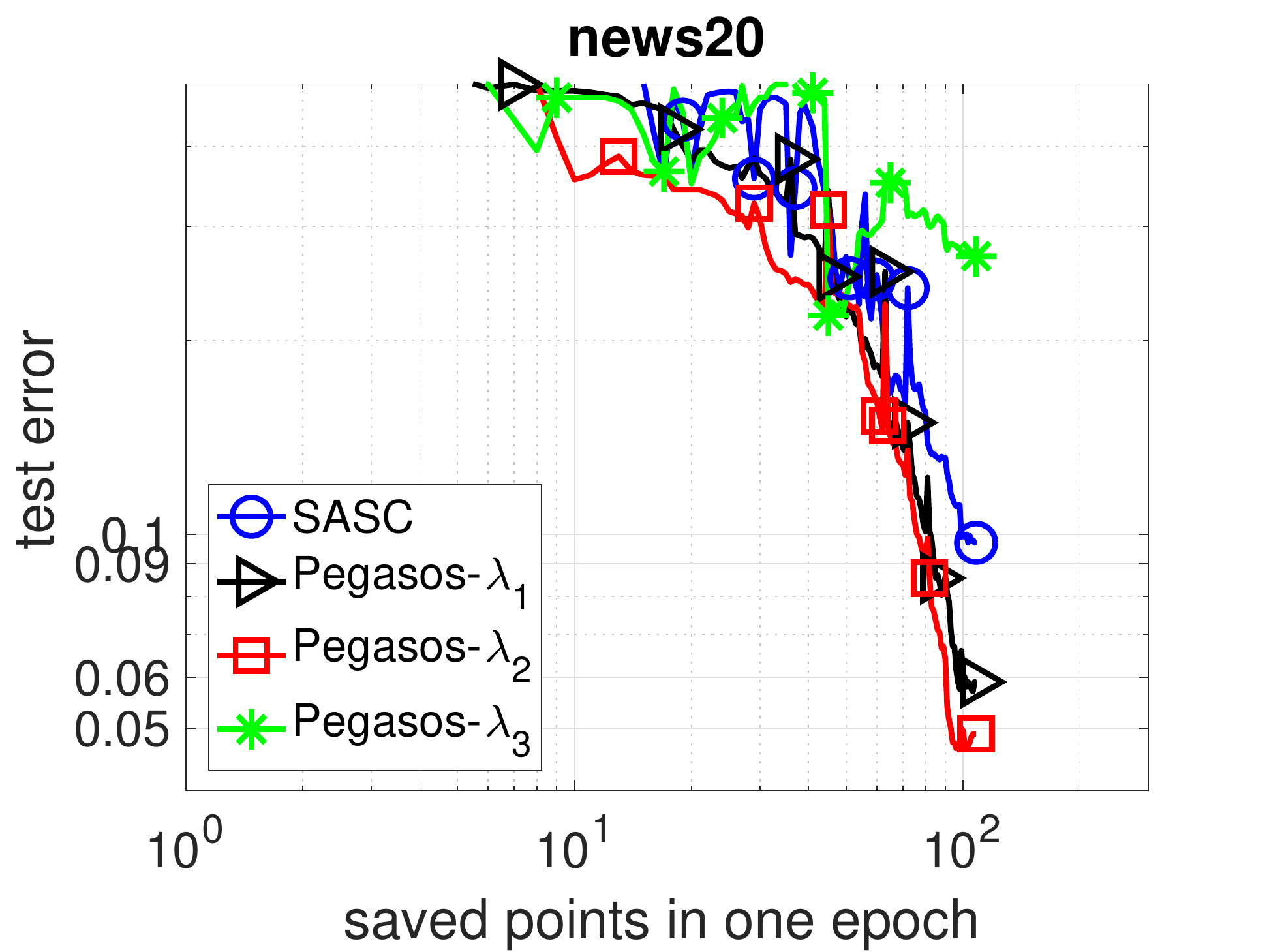}
\includegraphics[width=0.3\columnwidth]{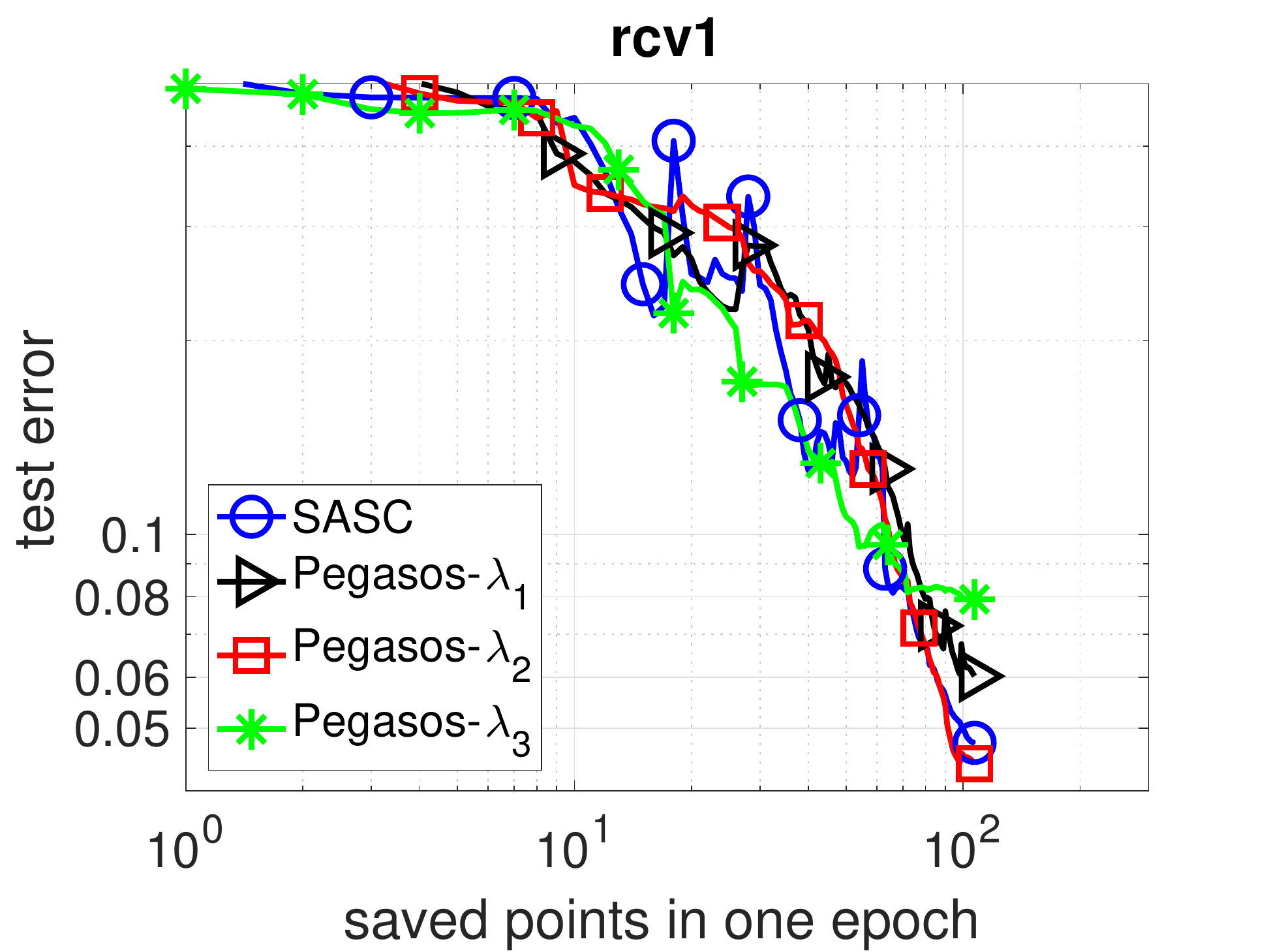}
\caption{Performance of SASC and Pegasos on SVM for three different datasets.}
\label{fig:svm}
\end{figure*}

We illustrate the performance of SASC and Pegasos in Figure~\ref{fig:svm}.
SASC seems to be comparable to Pegasos for different regularization parameters.
As can be seen in Figure~\ref{fig:svm}, Pegasos performs well for good selection of the regularization parameter.
However, when the parameter is selected incorrectly, it might stagnate at a high test error which can be observed in the plots.
On the other hand, SASC gets comparable, if not better, performance without the need to tune regularization parameter.

\bibliography{sgard-bibtex.bib}
\bibliographystyle{abbrvnat}

\newpage 
\section*{Appendix}
We wish to solve the following problem:

\begin{align*}
\min _{x\in\mathbb{R}^d} & \{ P(x) := \mathbb{E}\left[ f(x, \xi) \right] + h(x)\} \\
& A(\xi) x \in b(\xi)  ~~ \xi \text{-almost surely}
\end{align*}

where $F(x) = \mathbb{E} \left[ f(x, \xi) \right]$ has $L(\nabla F)$ Lipschitz gradient and $h(x)$ is nonsmooth but proximable. Let us define the following quantities which we will use in the sequel:
\begin{equation}
P_{\beta}(x) = F(x) + G_\beta(Ax) + h(x) = \mathbb{E}_{\xi} \left[ f(x, \xi) + g_{\beta}(A(\xi)x, b(\xi) ) \right] + h(x),
\end{equation}
where $g_{\beta}(A(\xi)x, b(\xi)) =\frac{1}{2\beta} \dist{\left( A(\xi)x, b(\xi) \right)}^2$. We also define $S_{\beta}(x) = P_{\beta}(x) - P(x_\star)$ which we refer to as the smoothed gap function. We first prove a lemma to relate the decrease of smoothed gap function to the objective suboptimality and primal feasibility.

\begin{replemma}{lem: smooth_gap_lemma}
Let $(x_\star, y_\star)$ be a saddle point of

\begin{equation}
\min_x \max_y P(x) + \int \langle A(\xi)x, y(\xi) \rangle  - \supp_{b(\xi)}(y(\xi)) \mu(d\xi),
\end{equation}

and $S_{\beta}(x) = P_{\beta}(x)-P(x_\star) = P(x) - P(x_\star) + \frac{1}{2\beta} \int \dist{(A(\xi)x, b(\xi))^2 \mu(d\xi)}$.
Then, the following hold:

\begin{align}
&S_{\beta}(x) \geq -\frac{\beta}{2}\|y_\star\|^2 \nonumber \\
&P(x) - P(x_\star) \geq -\frac{1}{4\beta} \int \dist{(A(\xi)x, b(\xi))^2}\mu(d\xi) - \beta \| y_\star \|^2 \nonumber \\
&P(x) - P(x_\star) \leq S_{\beta}(x) \nonumber \\
&\int \dist{(A(\xi)x, b(\xi))^2}\mu(d\xi) \leq 4\beta^2\| y_\star\|^2 + 4\beta S_{\beta}(x) \nonumber
\end{align}

\end{replemma}

\begin{proof}
We recall that the optimal Lagrange multiplier  $y_\star = (y_\star(\xi))_{\xi}$ is a random variable of $\mathcal Y$. 
It is indeed of bounded variance thanks to the constraint qualification condition we assumed~\cite{bauschke2011convex}. We start with:
\begin{equation}\label{eq: lm1_bound1}
-\int \langle A(\xi)x, y_\star(\xi)\rangle  +  \supp_{b(\xi)}{(y_\star(\xi))} \mu(d\xi) \leq P(x) - P(x_\star) = S_{\beta}(x) - \frac{1}{2\beta}\int \dist{(A(\xi)x, b(\xi))^2} \mu(d\xi),
\end{equation}
where the inequality is due to saddle point definition, and the equality is due to the definition of $S_{\beta}$.

We continue by bounding the inner product $\langle A(\xi)x, y_\star(\xi)\rangle$.
Let $z:=A(\xi)x$, then
\begin{align}
\langle z, y_{\star}(\xi) \rangle &= \langle z - \Pi_{b(\xi)}(z), y_\star(\xi) \rangle + \langle \Pi_{b(\xi)}(z), y_\star(\xi) \rangle \leq \dist(z, b(\xi)) \| y_\star(\xi)\| + \langle A(\xi)x_\star, y_\star(\xi) \rangle \nonumber \\
&\leq \frac{1}{4\beta} \dist{(z, b(\xi))^2} + \beta \|y_\star(\xi)\|^2 + \supp_{b(\xi)}(y_\star(\xi)), \label{eq: lm1_inner}
\end{align}
where the first inequality follows from Cauchy-Schwarz inequality, the optimality conditions, and properties of Fenchel's transform: $A(\xi)x_\star \in \partial \supp_{b(\xi)}(y_\star(\xi)) \iff y_\star(\xi) \in \partial \delta_{b(\xi)}(A(\xi)x_\star) \iff \langle p - A(\xi)x_\star, y_\star(\xi) \rangle \leq 0 $, for all $p \in b(\xi)$. The second inequality follows from $2ab \leq a^2 + b^2$ and $-\supp_{b(\xi)}(y_\star(\xi))+\langle A(\xi)x_\star, y_\star(\xi) \rangle = \inf_{q\in b(\xi)} \langle A(\xi)x_\star - q, y_\star(\xi)\rangle \leq 0$, since $A(\xi)x_\star \in b(\xi)$.

We now use  $\int \|y_\star(\xi)\|^2 \mu(d\xi) = \|y_\star\|^2$,   integrate~\eqref{eq: lm1_inner} and plug in to~\eqref{eq: lm1_bound1} to obtain last inequality. Second and third inequalities directly follow from~\eqref{eq: lm1_bound1} and~\eqref{eq: lm1_inner}.

For the first inequality:
\begin{align*}
S_{\beta}(x) &= P(x) + \frac{1}{2\beta} \int \dist(A(\xi)x, b(\xi))^2\mu(d\xi)- P(x_\star)   \\
&= P(x) - P(x_\star) + \int \max_{y \in \mathbb R^d} \langle A(\xi)x, y \rangle - \supp_{b(\xi)}(y) - \frac{\beta}{2}\norm{y}^2 \mu(d\xi)\\
& \geq P(x) - P(x_\star) + \int \langle A(\xi)x, y_\star(\xi) \rangle - \supp_{b(\xi)}(y_\star(\xi)) - \frac{\beta}{2}\norm{y_\star(\xi)}^2\mu(d\xi)
\geq -\frac{\beta}{2} \|y_\star\|^2,
\end{align*}
where the second equality follows from the definition of smoothing and the inequality is due to~\eqref{eq: lm1_bound1}. 
\end{proof}

\subsection{General Convex Case}

\begin{lemma}\label{lem: nonsc_smoothed_gap_lemma}
Assume that for all $s$, $\frac{L(\nabla F)+\|A\|^2_{2, \infty}/\beta_s}{2\alpha_s} \leq 0$, $2\alpha_s \|A\|^2_{2, \infty} - \frac{\beta_s}{2} \leq 0$ and \\ $\mathbb{E}\left[ \| \nabla f(x, \xi) - \nabla F(x)\|^2  \right] \leq \sigma_f^2$. Let $ \bar{x}^s = \frac{1}{m_s} \sum_{k=1}^{m_s} x_k^s$. Then,

\begin{equation}
\mathbb{E} \left[ P_{\beta_S}(\bar{x}^S) - P_{\beta_S}(x_\star) \right] \leq \frac{1}{2\alpha_S m_S} \| x_\star - x_0 ^0\|^2 + \frac{\sum_{s=0}^{S-1}\beta_s \alpha_s m_s}{2\alpha_S m_S} \| y_\star \|^2 + 2 \frac{\sum_{s=0}^S \alpha_s^2 m_s}{\alpha_S m_S} \sigma_f^2. 
\end{equation}

\end{lemma}

\begin{proof}
Let us define $z = Ax \in \mathcal Y$. 
We start by using Lipschitz gradient property of the function $f(x) + G_{\beta_s}(Ax)$
 
\begin{align}
P_{\beta_s}(x_{k+1}^s) &\leq F(x_k^s) + h(x_{k+1}^s) + G_{\beta_s}(Ax_k^s) + \langle \nabla F(x_k^s) + A^\top \nabla_z G_{\beta_s}(Ax_k^s), x^s_{k+1} - x^s_{k} \rangle \nonumber \\
&+ \frac{L(\nabla F + \nabla_x G_{\beta_s})}{2} \| x^s_{k+1} - x^s_k \|^2 \nonumber \\
&\leq F(x_k^s) + h(x_{k+1}^s) + G_{\beta_s}(x_k^s) + \langle \nabla f(x_k^s, \xi) + A(\xi)^\top \nabla _z g_{\beta_s}(A(\xi)x_k^s, \xi), x^s_{k+1} - x^s_{k} \rangle \nonumber\\ 
&+\langle \nabla F(x_k^s) - \nabla f(x_k^s, \xi) + A^\top \nabla_z G_{\beta_s}(Ax_k^s) - A(\xi)^\top \nabla _z g_{\beta_s}(A(\xi)x_k^s, \xi), x_{k+1}^s - x_k^s \rangle \nonumber \\
&+ \frac{L(\nabla F + \nabla_x G_{\beta_s})}{2} \| x^s_{k+1} - x^s_k \|^2. \label{eq: big_term1}
\end{align}

We will bound the linear terms in~\eqref{eq: big_term1} separately.

First, we use the three-point inequality (Property 1 from \cite{tseng2008accelerated}) with $x=x_\star$ to obtain,
\begin{align} \label{eq: lin_term1}
h(x_{k+1}^s) + &\langle \nabla f(x_k^s, \xi) + A(\xi)^\top \nabla _z g_{\beta_s}(A(\xi)x_k^s, \xi), x^s_{k+1} - x^s_{k} \rangle 
\leq h(x_\star) - \frac{1}{2\alpha_s} \| x_{k+1}^s - x_k^s \|^2 \\ &+ \langle \nabla f(x_k^s, \xi) + A(\xi)^\top \nabla _z g_{\beta_s}(A(\xi)x_k^s, \xi), x_\star - x^s_{k} \rangle + \frac{1}{2\alpha_s} \| x_\star - x_k^s \|^2 - \frac{1}{2\alpha_s} \| x_\star - x_{k+1}^s \|^2 \nonumber
\end{align}

Further, by the fact that $g_{\beta_s}(\cdot, \xi)$ has $1/\beta_s$-Lipschitz gradient, 
\begin{align}\label{eq: g_beta_lips}
\langle A(\xi)^\top \nabla_z g_{\beta_s}(A(\xi)x_k^s, \xi), x_\star - x_k^s \rangle &\leq g_{\beta_s}(A(\xi)x_\star, \xi) - g_{\beta_s}(A(\xi)x_k^s, \xi) \notag\\ &- \frac{\beta_s}{2} \| \nabla_z g_{\beta_s}(A(\xi)x_k^s, \xi) - \nabla_z g_{\beta_s}(A(\xi)x_\star, \xi) \|^2 \nonumber \\
&= g_{\beta_s}(A(\xi)x_\star, \xi) - g_{\beta_s}(A(\xi)x_k^s, \xi) - \frac{\beta_s}{2} \| \nabla_z g_{\beta_s}(A(\xi)x_k^s, \xi) \|^2,
\end{align}
where the equality follows from the fact that $\nabla_z g_{\beta_s}(A(\xi)x_\star, \xi) = 0$, due to the definition of $g_{\beta_s}(\cdot, \xi)$ and the fact that $A(\xi)x_\star \in b(\xi)$.

We now use the convexity, $ \langle \nabla f(x_k^s, \xi), x_\star - x_k^s \rangle \leq f(x_\star, \xi) - f(x_k^s, \xi) $ and~\eqref{eq: g_beta_lips} in~\eqref{eq: lin_term1} to get
\begin{align}
h(&x_{k+1}^s) + \langle \nabla f(x_k^s, \xi) + A(\xi)^\top \nabla _z g_{\beta_s}(A(\xi)x_k^s, \xi), x^s_{k+1} - x^s_{k} \rangle \leq h(x_\star) + f(x_\star, \xi) -f(x_k^s, \xi)  + g_{\beta_s}(A(\xi)x_\star, \xi) \nonumber \\ 
&- g_{\beta_s}(A(\xi)x_k^s, \xi) - \frac{\beta_s}{2} \| \nabla_z g_{\beta_s}(A(\xi)x_k^s, \xi) \|^2
+ \frac{1}{2\alpha_s} \| x_\star - x_k^s \|^2 - \frac{1}{2\alpha_s} \| x_\star - x_{k+1}^s \|^2 - \frac{1}{2\alpha_s} \| x_{k+1}^s - x_k^s \|^2 \label{eq: lin_term1_last}
\end{align}

We define
\begin{equation*}
T_{\alpha_s g}(x_k^s) = \text{prox}_{\alpha_s h} (x_k^s - \alpha_s (\nabla F(x_k^s) + A^\top \nabla _z G_{\beta_s}(Ax_k^s) )).
\end{equation*}
For the second linear term in~\eqref{eq: big_term1}, we apply conditional expectation knowing $x_k^s$, with respect to the choice of $\xi=\xi_{k+1}$, 
\begin{align}
&\mathbb{E}_k \left[ \langle \nabla F(x_k^s) - \nabla f(x_k^s, \xi) + A^\top \nabla_z G_{\beta_s}(Ax_k^s) - A(\xi)^\top \nabla _z g_{\beta_s}(A(\xi)x_k^s, \xi), x_{k+1}^s - x_k^s \rangle \right] = \nonumber \\
& \mathbb{E}_k \Big[ \langle \nabla F(x_k^s) - \nabla f(x_k^s, \xi) + A^\top \nabla_z G_{\beta_s}(Ax_k^s) - A(\xi)^\top \nabla _z g_{\beta_s}(A(\xi)x_k^s, \xi), x_{k+1}^s - T_{\alpha_s g}(x_k^s) \rangle \nonumber \\
&+ \langle \nabla F(x_k^s) - \nabla f(x_k^s, \xi) + A^\top \nabla_z G_{\beta_s}(Ax_k^s) - A(\xi)^\top \nabla _z g_{\beta_s}(A(\xi)x_k^s, \xi), T_{\alpha_s g}(x_k^s) - x_k^s \rangle \Big] \nonumber \\
&= \mathbb{E}_k \left[ \langle \nabla F(x_k^s) - \nabla f(x_k^s, \xi) + A^\top \nabla_z G_{\beta_s}(Ax_k^s) - A(\xi)^\top \nabla _z g_{\beta_s}(A(\xi)x_k^s, \xi), x_{k+1}^s - T_{\alpha_s g}(x_k^s) \rangle \right] \nonumber \\
&\leq \mathbb{E}_k \left[ \| \nabla F(x_k^s) - \nabla f(x_k^s, \xi) + A^\top \nabla_z G_{\beta_s}(Ax_k^s) - A(\xi)^\top \nabla _z g_{\beta_s}(A(\xi)x_k^s, \xi) \| \|x_{k+1}^s - T_{\alpha_s g}(x_k^s) \| \right] \nonumber \\
&\leq \alpha_s\mathbb{E}_k \left[ \| \nabla F(x_k^s) - \nabla f(x_k^s, \xi) + A^\top \nabla_z G_{\beta_s}(Ax_k^s) - A(\xi)^\top \nabla _z g_{\beta_s}(A(\xi)x_k^s, \xi) \|^2 \right] \nonumber \\
&\leq 2\alpha_s\mathbb{E}_k \left[ \| \nabla F(x_k^s) - \nabla f(x_k^s, \xi) \|^2 \right] + 2\alpha_s \mathbb{E}_k \left[ \| A^\top \nabla_z G_{\beta_s}(Ax_k^s) - A(\xi)^\top \nabla _z g_{\beta_s}(A(\xi)x_k^s, \xi) \|^2 \right] \nonumber \\
&\leq 2\alpha_s\sigma_f^2 + 2\alpha_s \mathbb{E}_k \left[ \| A(\xi)^\top \nabla _z g_{\beta_s}(A(\xi)x_k^s, \xi) \|^2 \right] \nonumber \\
&\leq 2\alpha_s \sigma_f^2 + 2\alpha_s \sup_{\xi} \|A(\xi)\|^2 \mathbb{E}_k \left[ \| \nabla _z g_{\beta_s}(A(\xi)x_k^s, \xi) \|^2 \right], \label{eq: error_terms}
\end{align}
where the second inequality is due to the definition of $x_{k+1}^s$, $T_{\alpha_s g}(x_k^s)$ and nonexpansiveness of proximal operator. Fourth inequality is due to the fact that $\mathbb{E}\left[ \| X - \mathbb{E}\left[ X \right] \|^2 \right] = \mathbb{E}\left[ \| X \|^2 \right] -  \left( \mathbb{E} \left[ X \right] \right)^2$, for any random variable X and $\mathbb{E}_k \left[ A(\xi)^\top \nabla _z g_{\beta_s}(A(\xi)x_k^s, \xi) \right] =   A^\top \nabla_z G_{\beta_s}(Ax_k^s)$.

We take conditional expectation of~\eqref{eq: big_term1}, knowing $x_k^s$, and plug in~\eqref{eq: lin_term1_last},~\eqref{eq: error_terms} to obtain
\begin{align*}
\mathbb{E}_k \left[ P_{\beta_s}(x_{k+1}^s)\right] &\leq \mathbb{E}_k \bigg[ F(x_k^s) + h(x_{k+1}^s) + G_{\beta_s}(x_k^s) + \langle \nabla f(x_k^s, \xi) + A(\xi)^\top \nabla _z g_{\beta_s}(A(\xi)x_k^s, \xi), x^s_{k+1} - x^s_{k} \rangle \nonumber\\ 
&+\langle \nabla F(x_k^s) - \nabla f(x_k^s, \xi) + A^\top \nabla_z G_{\beta_s}(Ax_k^s) - A(\xi)^\top \nabla _z g_{\beta_s}(A(\xi)x_k^s, \xi), x_{k+1}^s - x_k^s \rangle \nonumber \\ 
&+ \frac{L(\nabla F + \nabla_x G_{\beta_s})}{2} \| x^s_{k+1} - x^s_k \|^2 \bigg] \nonumber\\
& \leq  P_{\beta_s}(x_\star) + \frac{1}{2\alpha_s} \| x_\star - x_k^s \|^2 - \frac{1}{2\alpha_s} \mathbb{E}_k \left[ \|x_\star - x_{k+1}^s \|^2 \right] \\ &+ \left( 2\alpha_s \|A\|^2_{2, \infty} - \frac{\beta_s}{2} \right) \mathbb{E}_k \left[ \nabla_z g_{\beta_s}(A(\xi)x_k^s, \xi)  \right] \nonumber \\
&+ \left( \frac{L(\nabla F) + \|A\|^2_{2, \infty}/{\beta_s}}{2} - \frac{1}{2\alpha_s} \right) \mathbb{E}_k \left[ \| x_{k+1}^s - x_k^s \|^2 \right] + 2\alpha_s \mathbb{E}_k \left[ \| \nabla f(x_k^s, \xi) \|^2 \right].
\end{align*}
We use the assumptions that $2\alpha_s \|A\|^2_{2, \infty} - \frac{\beta_s}{2} \leq 0$ and $\frac{L(\nabla F) + \|A\|^2_{2, \infty}/{\beta_s}}{2} - \frac{1}{2\alpha_s} \leq 0$ to get
\begin{align*}
\mathbb{E}_k \left[ P_{\beta_s}(x_{k+1}^s)\right] &\leq  P_{\beta_s}(x_\star) + \frac{1}{2\alpha_s} \| x_\star - x_k^s \|^2 - \frac{1}{2\alpha_s} \mathbb{E}_k \left[ \|x_\star - x_{k+1}^s \|^2 \right] + 2\alpha_s \sigma_f^2.
\end{align*}

We apply total expectation with respect to the history $\mathcal{F}_{k}=\{ \xi_0, \dots, \xi_k \}$ and sum for $k\in\{0, 1, \dots, m_s-1\}$ to obtain

\begin{align}
\mathbb{E} \left[ P_{\beta_s} \left(\frac{1}{m_s} \sum_{k=1}^{m_s} x_{k}^s \right) - P_{\beta_s}(x_\star) \right] &\leq \frac{1}{2\alpha_s m_s} \mathbb{E} \left[ \| x_\star - x_0^s \|^2 \right] - \frac{1}{2\alpha_s m_s} \mathbb{E} \left[ \|x_\star - x_{m_s}^s \|^2 \right] \nonumber \\
&+ \frac{2\alpha_s}{m_s} \sum_{k=0}^{m_s-1} \sigma_f^2\nonumber\\
&\leq \frac{1}{2\alpha_s m_s} \mathbb{E} \left[ \| x_\star - x_0^s \|^2 \right] - \frac{1}{2\alpha_s m_s} \mathbb{E} \left[ \|x_\star - x_{m_s}^s \|^2 \right] + 2\alpha_s \sigma_f^2. \label{eq: inner_loop_rec1}
\end{align}

By Lemma~\ref{lem: smooth_gap_lemma}, we know that for all $x$, we have (we use that $P_{\beta_s}(x_\star) = P(x_\star)$):
\begin{align}\label{eq: smoothed_obj_lowerbound}
P_{\beta_s}(x) - P_{\beta_s}(x_\star) \geq -\frac{\beta_s}{2} \| y_\star\|^2.
\end{align}
By the restarting rule of the inner loop, one has $x_{m_s}^s = x_0^{s+1}$. Using~\eqref{eq: smoothed_obj_lowerbound} in~\eqref{eq: inner_loop_rec1}, we obtain
\begin{align} \label{eq: iter_rec1}
\mathbb{E} \left[ \| x_\star - x_0^{s+1} \|^2 \right] \leq \mathbb{E} \left[ \| x_\star - x_0^{s} \|^2 \right]  +\beta_s \alpha_s m_s \| y_\star \|^2 + 4\alpha_s^2 m_s \sigma_f^2
\end{align}
We now sum~\eqref{eq: iter_rec1} for $s\in\{ 0, 1, \dots, S-1 \}$
\begin{align} \label{eq: iter_rec2}
\mathbb{E} \left[ \| x_\star - x_0^{S} \|^2 \right] \leq \| x_\star - x_0^{0} \|^2  + \sum_{s=0}^{S-1} \beta_s \alpha_s m_s \| y_\star \|^2 + 4\sum_{s=0}^{S-1}\alpha_s^2 m_s \sigma_f^2
\end{align}
We now use~\eqref{eq: iter_rec2} in~\eqref{eq: inner_loop_rec1} to obtain
\begin{align*}
\mathbb{E} \left[ P_{\beta_S} \left(\bar{x}^S \right) - P_{\beta_S}(x_\star) \right] &\leq \frac{1}{2\alpha_S m_S} \| x_\star - x_0^0 \|^2 + \frac{\sum_{s=0}^{S-1}\beta_s \alpha_s m_s}{2\alpha_S m_S} \| y_\star \|^2 + 4 \frac{\sum_{s=0}^{S-1} \alpha_s^2 m_s}{2\alpha_S m_S} \sigma_f^2 + 2\alpha_S \sigma_f^2
\end{align*}
\end{proof}

In the following lemma, we estimate the rates of the parameters to determine the final convergence rates:

\begin{lemma}\label{lem: nonsc_param_lemma}
Denote as $M_S = \sum_{s=0}^{S}m_s$ the total number of iterations to compute $\bar{x}^S$. Let $\omega > 1$. Let us choose $\alpha_0 \leq \frac{3}{4L(\nabla f)}$, $m_0 \in \mathbb{N}_\ast$, $m_s = \lfloor m_0 \omega^S \rfloor$, $\alpha_s = \alpha_0 \omega^{-s/2}$ and $\beta_s = 4\alpha_s \|A\|^2_{2, \infty}$. Then, for all $s$, $\frac{L(\nabla F)+\|A\|^2_{2, \infty}/{\beta_s}}{2} - \frac{1}{2\alpha_s} \leq 0$ and $2\alpha_s \|A\|^2_{2, \infty} - \frac{\beta_s}{2} \leq 0$. Moreover,

\begin{align}
&\beta_s \leq 4\alpha_0 \sqrt{m_0}\|A\|^2_{2, \infty} \frac{\sqrt{\omega}}{\sqrt{\omega - 1}} \frac{1}{\sqrt{M_s}} \nonumber \\
&\alpha_s m_s \geq \alpha_0 \frac{(m_0-1)}{\sqrt{m_0}}\frac{\sqrt{\omega-1}}{\sqrt{\omega}}\sqrt{M_s} \nonumber \\
&\sum_{s=0}^{S-1}\beta_s \alpha_s m_s \leq 4\alpha_0^2 \|A\|^2_{2, \infty} m_0 \frac{\log(M_s/m_0)}{\log(\omega)} \nonumber \\
&\sum_{s=0}^{S} \alpha_s^2 m_s \leq \alpha_0 m_0 \left(\frac{\log(M_s/m_0)}{\log(\omega)}+1\right) \nonumber
\end{align}

\begin{proof}
By definition of $\beta_s$, $2\alpha_s \|A\|^2_{2, \infty} - \frac{\beta_s}{2} \leq 0$ holds with equality.
By using the definition of $\beta_s$, the fact that $\alpha_s$ is a decreasing sequence and the condition on $\alpha_0$, we have $\frac{L(\nabla F)+\|A\|^2_{2, \infty}/{\beta_s}}{2} - \frac{1}{2\alpha_s} \leq 0$.

We now compute the total number of iterations:
\begin{equation}
M_S = \sum_{s=0}^{S}m_s = \sum_{s=0}^{S} \lfloor m_0 \omega^s \rfloor \leq \sum_{s=0}^{S} m_0 \omega^s = m_0 \frac{\omega^{S+1}-1}{\omega - 1},
\end{equation}
which in turn gives
\begin{equation}
\omega^S \geq \frac{\omega - 1}{\omega}\frac{M_S}{m_0} + \frac{1}{\omega} \geq \frac{\omega-1}{\omega} \frac{M_S}{m_0}.
\end{equation}
We now use this bound to get
\begin{align*}
&\beta_S = 4\alpha_S\|A\|^2_{2, \infty} = 4\alpha_0\|A\|^2_{2, \infty} \omega^{-S/2} \leq 4\alpha_0 \|A\|^2_{2, \infty} \frac{\sqrt{\omega}}{\sqrt{\omega-1}}\frac{\sqrt{m_0}}{\sqrt{M_S}} \\
&\alpha_S m_S = \alpha_0 \omega^{-S/2} \lfloor m_0 \omega^S \rfloor \geq \alpha_0 m_0\omega^{S/2} - \alpha_0\omega^{-S/2} \geq \alpha_0 \frac{(m_0-1)}{\sqrt{m_0}} \frac{\sqrt{\omega-1}}{\sqrt{\omega}}\sqrt{M_S}.
\end{align*}
We can also lower bound $M_S$ as
\begin{equation*}
M_S = \sum_{s=0}^S m_s = \sum_{s=0}^S \lfloor m_0 \omega^s \rfloor = m_0 + \sum_{s=1}^S \lfloor m_0 \omega^s \rfloor \geq m_0 + m_0 \omega^S - 1\geq m_0\omega^S,
\end{equation*}
since $m_0 \geq 1$. We thus get
\begin{equation}\label{eq: s_bound}
S \leq \frac{\log{(M_S/m_0)}}{\log(\omega)}
\end{equation}
Further,
\begin{equation*}
\beta_s \alpha_s m_s = 4\alpha_0^2 \|A\|^2_{2, \infty} \omega^{-s}\lfloor m_0\omega^s \rfloor \leq 4\alpha_0^2 \|A\|^2_{2, \infty} m_0.
\end{equation*}
Now we use~\eqref{eq: s_bound} to show that
\begin{equation*}
\sum_{s=0}^{S-1} \beta_s\alpha_s m_s \leq S \times 4\alpha_0^2 \|A\|^2_{2, \infty}m_0\leq 4\alpha_0^s\|A\|^2_{2, \infty} m_0 \frac{\log(M_S/m_0)}{\log(\omega)}.
\end{equation*}

Lastly, we use the relation $\beta_s = 4\alpha_s\|A\|^2_{2, \infty}$ to conclude last bound.
\end{proof}

\end{lemma}

\begin{reptheorem}{th: nonsc_th}
Assume $F$ is convex and $L(\nabla F)$ smooth, 
and $\exists \sigma_f$ such that $\mathbb E[\norm{\nabla f(x,\xi) - \nabla F(x)}^2] \leq \sigma_f^2$.
Denote $M_S = \sum_{s=0}^S m_s$. Let us set $\omega > 1$, $\alpha_0 \leq \frac{3}{4L(\nabla f)}$, $m_0 \in \mathbb{N}_\ast$, $m_s = \lfloor m_0 \omega^s \rfloor$, and $\beta_s =4\alpha_s \| A\|^2_{2, \infty}$. Then, for all S,

\begin{align*}
&\mathbb E[P(\bar{x}^S) - P(x_\star)] \leq \frac{C_1}{\sqrt{M_S}} \left[ C_2 + \frac{\log(M_S/m_0)}{\log(\omega)}C_3 \right] \\
& \mathbb E[P(\bar{x}^S) - P(x_\star)] \geq
  -\frac{2C_4}{\sqrt{M_S}}\|y_\star\|^2 - \frac{C_1}{\sqrt{M_S}} \left[ C_2 + \frac{\log(M_S/m_0)}{\log(\omega)}C_3 \right] \\
&\sqrt{\mathbb{E}\left[ \dist(A(\xi)\bar{x}^S, b(\xi))^2 \right]} \leq 
 \frac{1}{\sqrt{M_S}} \left[ 2C_4\|y_\star\| + 2 \sqrt{C_1 C_4}\sqrt{C_2 + \frac{\log(M_S/m_0)}{\log(\omega)}C_3} \right]
\end{align*}
where $C_1 = \frac{\sqrt{m_0\omega}}{\alpha_0 (m_0 -1)\sqrt{\omega-1}}$, $C_2 = \frac{\|x_\star - x_0^0\|^2}{2}+2\alpha_0 m_0 \sigma_f^2$, $C_3 = 2\alpha_0^2 \|A\|_{2,\infty}^2 m_0 \|y_\star \|^2 + 2\alpha_0 m_0 \sigma_f^2$ and  $C_4 = 4 \alpha_0 \sqrt{m_0}\|A\|_{2,\infty}^2 \frac{\sqrt{\omega}}{\sqrt{\omega-1}}$.
\end{reptheorem}

\begin{proof}
We first combine Lemma~\ref{lem: nonsc_smoothed_gap_lemma} and Lemma~\ref{lem: nonsc_param_lemma}:
\begin{align*}
\mathbb E[&S_{\beta_S}(\bar x^S)] =\mathbb{E} \left[ P_{\beta_S}(\bar{x}^S) - P_{\beta_S}(x_\star) \right] \leq \frac{1}{2\alpha_S m_S} \| x_\star - x_0 ^0\|^2 + \frac{\sum_{s=0}^{S-1}\beta_s \alpha_s m_s}{2\alpha_S m_S} \| y_\star \|^2 + 2 \frac{\sum_{s=0}^S \alpha_s^2 m_s}{\alpha_S m_S} \sigma_f^2 \\
&\leq \frac{\frac{\sqrt{m_0}}{(m_0-1)}\frac{\sqrt{\omega}}{\sqrt{\omega-1}}}{\alpha_0 \sqrt{M_s}} \Big[
\frac{1}{2} \| x_\star - x_0 ^0\|^2 + \frac{4\alpha_0^2 \|A\|^2_{2, \infty} m_0 \frac{\log(M_s/m_0)}{\log(\omega)}}{2} \| y_\star \|^2 + 2 \alpha_0 m_0 \left(\frac{\log(M_s/m_0)}{\log(\omega)}+1\right) \sigma_f^2 \Big] \\
& = \frac{C_1}{\sqrt{M_S}} \Big[C_2 + \frac{\log(M_S/m_0)}{\log(\omega)}C_3 \Big]
\end{align*}

We combine the inequality above with the bound $\beta_S \leq 4\alpha_0 \sqrt{m_0}\|A\|^2_{2, \infty} \frac{\sqrt{\omega}}{\sqrt{\omega - 1}} \frac{1}{\sqrt{M_s}} = \frac{C_4}{\sqrt{M_S}}$ and Lemma~\ref{lem: smooth_gap_lemma}:
\[
\sqrt{\mathbb{E}\left[ \dist(A(\xi)\bar{x}^s, b(\xi))^2 \right]} \leq
\sqrt{4 \beta_S^2 \norm{y_\star}^2 + 4 \beta_S S_{\beta_S}(\bar x^S)} \leq \frac{2C_4\norm{y_\star}}{\sqrt{M_S}} + \frac{2 \sqrt{C_1C_4}}{\sqrt{M_S}} \sqrt{C_2 + \frac{\log(M_S/m_0)}{\log(\omega)}C_3}
\]
The other inequalities follow similarly using
\[
S_{\beta_S}(\bar x^S) \geq P(\bar x^S) - P(x_\star) \geq -\frac{1}{4\beta_S} \int \dist{(A(\xi)\bar x^S, b(\xi))^2}\mu(d\xi) - \beta_S \| y_\star \|^2 \geq -2\beta_S \| y_\star \|^2 - S_{\beta_S}(\bar x^S)
\]
\end{proof}

\subsection{Restricted Strongly Convex Case}

\begin{lemma}\label{lem: sc_sm_gap}
Assume that $F$ is convex and $L(\nabla F)$-smooth, $P$ is 
$\mu$-restricted strongly convex and $\mathbb E[\norm{\nabla f(x, \xi) - \nabla F(x)}^2] \leq \sigma_f^2$ for all $x$.
Assume that for all $s$, $\frac{L(\nabla F)+\|A\|^2_{2, \infty}/\beta_s}{2\alpha_s} \leq 0$, $2\alpha_s \|A\|^2_{2, \infty} - \frac{\beta_s}{2} \leq 0$ and $\mu \alpha_s m_s \geq \frac{1}{c}$, for $c<1$. Let $ \bar{x}^S = \frac{1}{m_S} \sum_{k=1}^{m_S} x_k^S$. Then,
\begin{equation}
\mathbb{E} \left[ P_{\beta_S}(x_k^S) - P_{\beta_S}(x_\star) \right] \leq \frac{c^S}{2\alpha_S m_S}  \| x_\star - x_0^0 \|^2 + \frac{\sum_{s=0}^{S-1} c^{S+1-s} \beta_s \alpha_s m_s}{2\alpha_S m_S} \| y_\star \|^2 + \frac{\sum_{s=0}^{S-1} 4c^{S+1-s} \alpha_s^2 m_s}{2\alpha_S m_S} \sigma_f^2 +2\alpha_S \sigma_f^2.
\end{equation}
\end{lemma}

\begin{proof}
We proceed same as the proof of Lemma 4.1, until~\eqref{eq: inner_loop_rec1}. In the case where $F(x)+h(x)$ satisfies restricted strong convexity, instead of~\eqref{eq: smoothed_obj_lowerbound}, we can derive
\begin{align}\label{eq: smoothed_obj_lowerbound_sc}
P_{\beta_s}(x) - P_{\beta_s}(x_\star) \geq -\frac{\beta_s}{2} \| y_\star\|^2 + \frac{\mu}{2} \| x - x_\star \|^2.
\end{align}

We use~\eqref{eq: smoothed_obj_lowerbound_sc} in~\eqref{eq: inner_loop_rec1}, along with the restarting rule $\bar{x}^s = x_0^{s+1}$ to get
\begin{align}
\mu\alpha_s m_s \mathbb{E} \left[ \| x_\star - x_0^{s+1}\|^2 \right] \leq \mathbb{E} \left[ \| x_\star - x_0^s \|^2 \right] + \beta_s \alpha_s m_s \| y_\star \|^2 + 4\alpha_s^2 m_s \sigma_f^2.
\end{align}

Further, since $\mu\alpha_s m_s \geq \frac{1}{c}$, for $c<1$:
\begin{align}
\mathbb{E} \left[ \| x_\star - x_0^{s+1}\|^2 \right] \leq c\mathbb{E} \left[ \| x_\star - x_0^s \|^2 \right] +  c \beta_s \alpha_s m_s \| y_\star \|^2 + 4c\alpha_s^2 m_s \sigma_f^2.
\end{align}

We now get, by recursively applying the inequality for $s\in\{ 0, 1, \dots, S-1\}$
\begin{align}\label{eq: sc_iter_rec_final}
\mathbb{E} \left[ \| x_\star - x_0^{S}\|^2 \right] \leq c^S \| x_\star - x_0^0 \|^2 + \sum_{s=0}^{S-1} c^{S-s} \beta_s \alpha_s m_s \| y_\star \|^2 + \sum_{s=0}^{S-1} 4c^{S-s} \alpha_s^2 m_s \sigma_f^2.
\end{align}

We plug~\eqref{eq: sc_iter_rec_final} into~\eqref{eq: inner_loop_rec1} to obtain
\begin{align}
\mathbb{E} \left[ P_{\beta_S}(\bar{x}_k^S) - P_{\beta_S}(x_\star) \right] \leq \frac{c^S}{2\alpha_S m_S}  \| x_\star - x_0^0 \|^2 + \frac{\sum_{s=0}^{S-1} c^{S-s} \beta_s \alpha_s m_s}{2\alpha_S m_S} \| y_\star \|^2 + \frac{\sum_{s=0}^{S-1} 4c^{S-s} \alpha_s^2 m_s}{2\alpha_S m_S} \sigma_f^2 +2\alpha_S \sigma_f^2.
\end{align}
\end{proof}

In the following lemma, we estimate the rates of the parameters:
\begin{lemma}\label{lem: sc_param_lemma}
Denote as $M_S = \sum_{s=0}^{S}m_s$ the total number of iterations to compute $\bar{x}^S$. Let $\omega > 1$. Let us choose $\alpha_0 \leq \frac{3}{4L(\nabla f)}$, $m_0 \geq \frac{\omega}{\mu\alpha_0}$, $m_s = \lfloor m_0 \omega^S \rfloor$, $\alpha_s = \alpha_0 \omega^{-s}$, $\beta_s = 4\alpha_s \|A\|^2_{2, \infty}$ and $c=\frac{1}{\omega} < 1$. Then, for all $s$, $\frac{L(\nabla F)+\|A\|^2_{2, \infty}/{\beta_s}}{2} - \frac{1}{2\alpha_s} \leq 0$ and $2\alpha_s \|A\|^2_{2, \infty} - \frac{\beta_s}{2} \leq 0$. Moreover,

\begin{align*}
&\beta_s \leq 4\alpha_0 {m_0}\|A\|^2_{2, \infty} \frac{{\omega}}{{\omega - 1}} \frac{1}{{M_s}} \\
&\alpha_s m_s \geq \alpha_0 (m_0-1) \\
&\sum_{s=0}^{S-1} c^{S-s} \beta_s \alpha_s m_s \leq  4c^{S} \alpha_0^2 \|A\|^2_{2, \infty} m_0 \left( \frac{\log(M_s/m_0)}{\log(\omega)}\right)\\
&\sum_{s=0}^{S-1} c^{S-s} \alpha_s^2 m_s \leq c^{S} \alpha_0^2 m_0 \left( \frac{\log(M_s/m_0)}{\log(\omega)}\right) \\
&c^S \leq \frac{\omega}{\omega-1}\frac{m_0}{M_S} 
\end{align*}
\end{lemma}

\begin{proof}
We skip the proofs for the parts of the lemma that are the same as Lemma~\ref{lem: nonsc_param_lemma}.

We have
\begin{equation*}
\beta_s = 4\alpha_0\|A\|^2_{2, \infty} \omega^{-s} \leq 4\alpha_0 m_0 \|A\|^2_{2, \infty} \frac{\omega}{\omega-1}\frac{1}{M_s}.
\end{equation*}
In addition,
\begin{equation*}
\alpha_s m_s = \alpha_0 \omega^{-s} \lfloor m_0 \omega^s \rfloor \geq \alpha_0 \omega^{-s} (m_0 \omega^s -1) \geq \alpha_0 (m_0-1),
\end{equation*}
where the last inequality follows since $\omega^s \geq 1$.

We have
\begin{align*}
\sum_{s=0}^{S-1}c^{S-s}\beta_s \alpha_s m_s \leq \alpha_0 m_0 \sum_{s=0}^{S-1}c^{S-s}\beta_s \leq 4\alpha_0^2 \|A\|^2_{2, \infty} m_0 c^{S} \sum_{s=0}^{S-1} (\omega c)^{-s} = S \times 4\alpha_0^2 \|A\|^2_{2, \infty} m_0 c^{S}
\end{align*}

Next, we have $c^S = \omega^{-S} \leq \frac{\omega}{\omega-1}\frac{m_0}{M_S}$.

Fourth bound directly follows by combining the third bound with $\beta_s=4\alpha_s\|A\|^2_{2, \infty}$.
\end{proof}

\begin{reptheorem}{th: sc_th}
	Assume $F$ is convex and $L(\nabla F)$ smooth, $P$ is $\mu$-restricted strongly convex
	and $\exists \sigma_f$ such that $\mathbb E[\norm{\nabla f(x,\xi) - \nabla F(x)}^2] \leq \sigma_f^2$.
	Denote $M_S = \sum_{s=0}^S m_s$. Let us set $\omega > 1$, $\alpha_0 \leq \frac{3}{4L(\nabla f)}$, $m_0 \geq \frac{\omega}{ \mu\alpha_0}$, $m_s = \lfloor m_0 \omega^s \rfloor$, and $\beta_s =4\alpha_s \| A\|^2_{2, \infty}$. Then, for all S,
	\begin{align*}
&\mathbb E[P(\bar{x}^s) - P(x_\star)]\leq  \frac{1}{M_s} \left[ D_1 + \frac{\log(M_s/m_0)}{\log(\omega)} D_2 \right] \\
&\mathbb E[P(\bar{x}^s) - P(x_\star)] \geq -\frac{2 D_3}{M_s}\|y_\star\|^2 - \frac{1}{M_s} \left[ D_1 + \frac{\log(M_s/m_0)}{\log(\omega)}D_2 \right] \\
&\sqrt{\mathbb{E}\left[ \dist(A(\xi)\bar{x}^s, b(\xi))^2 \right]} \leq \frac{1}{{M_s}} \left[ 2D_3\|y_\star\| + 2 \sqrt{D_3} \sqrt{D_1 + \frac{\log(M_s/m_0)}{\log(\omega)}D_2} \right]
\end{align*}
where $D_1 = \frac{\omega}{\omega-1}\frac{m_0}{\alpha_0(m_0-1)}\frac 12 \norm{x_0^0 - x_\star}^2 + 2 \alpha_0 m_0  \frac{\omega}{\omega-1} \sigma_f^2$, $D_2 = \frac{2m_0^2\alpha_0\omega}{(m_0-1)(\omega-1)}\Big(\|A\|^2_{2, \infty}\|y_\star\|^2 + \sigma_f^2\Big)$,
$D_3 = 4 \alpha_0 {m_0}\|A\|^2_{2, \infty} \frac{{\omega}}{{\omega - 1}}$.
\end{reptheorem}

\begin{proof}
We first combine Lemma~\ref{lem: sc_sm_gap} and Lemma~\ref{lem: sc_param_lemma}:
\begin{align*}
&\mathbb{E} \left[ S_{\beta_S}(x_k^S)\right] \leq \frac{c^S}{2\alpha_S m_S}  \| x_\star - x_0^0 \|^2 + \frac{\sum_{s=0}^{S-1} c^{S-s} \beta_s \alpha_s m_s}{2\alpha_S m_S} \| y_\star \|^2 + \frac{\sum_{s=0}^{S-1} 4c^{S-s} \alpha_s^2 m_s}{2\alpha_S m_S} \sigma_f^2 +2\alpha_S \sigma_f^2 \\
& \leq \frac{\frac{\omega}{\omega-1}\frac{m_0}{M_S} }{\alpha_0 (m_0-1)} \Big[ \frac 12\| x_\star - x_0^0 \|^2 + \frac{4 \alpha_0^2 \|A\|^2_{2, \infty} m_0 \left( \frac{\log(M_s/m_0)}{\log(\omega)}\right)}{2} \| y_\star \|^2 + 2 \alpha_0^2 m_0 \left( \frac{\log(M_s/m_0)}{\log(\omega)}\right)\sigma_f^2 \Big] +\frac{\beta_S}{2\norm{A}^2_{2, \infty}} \sigma_f^2\\
& \leq \frac{1}{M_S}\Big[D_1 +  \frac{\log(M_s/m_0)}{\log(\omega)} D_2\Big]
\end{align*}
where $\beta_s \leq 4\alpha_0 {m_0}\|A\|^2_{2, \infty} \frac{{\omega}}{{\omega - 1}} \frac{1}{{M_s}} = \frac{D_3}{M_S}$.

We then use Lemma~\ref{lem: smooth_gap_lemma}.
\[
\sqrt{\mathbb{E}\left[ \dist(A(\xi)\bar{x}^s, b(\xi))^2 \right]} \leq
\sqrt{4 \beta_S^2 \norm{y_\star}^2 + 4 \beta_S S_{\beta_S}(\bar x^S)} \leq \frac{2D_3\norm{y_\star}}{M_S} + \frac{2 \sqrt{D_3}}{M_S} \sqrt{D_1 + \frac{\log(M_S/m_0)}{\log(\omega)}D_2}\;.
\]
The other inequalities follow similarly.
\end{proof}

\end{document}